\documentclass[a4paper,reqno,11pt]{amsart}
\usepackage{subfigure}

\usepackage{hyperref}
\usepackage{epstopdf}
\usepackage{color}
\usepackage{amssymb}
\usepackage{amsfonts}
\usepackage{amsmath}
\usepackage{mathrsfs}
\usepackage{dutchcal}
\usepackage{tikz}
\usepackage{float}

\newtheorem{theorem}{Theorem}[section]
\newtheorem{lemma}[theorem]{Lemma}

\theoremstyle{definition}

\newtheorem{remark}[theorem]{Remark}

\numberwithin{equation}{section}

\begin{document}
	\title[Elliptic inequalities]{Sharp Liouville type results for semilinear elliptic inequalities involving gradient terms on weighted graphs}
	\begin{abstract}
		
		 We study  nonexistence and existence of nontrivial positive solutions to the following semilinear elliptic
		inequality involving gradient terms
		\[
		\Delta u+u^p\left|\nabla u\right|^q\leq0,
		\]
		on weighted graphs, where $(p,q)\in\mathbb{R}^2$. We give a complete classification of $(p,q)$ under which sharp volume growth
assumptions
are established.
	\end{abstract}
	
	\author[Hao]{Lu Hao}
	\address{School of Mathematical Sciences and LPMC, Nankai University, 300071 Tianjin, P. R. China}
	\email{h18742512512@163.com}
	
	\author[Sun]{Yuhua Sun}
	\address{School of Mathematical Sciences and LPMC, Nankai University, 300071 Tianjin, P. R. China}
	\email{sunyuhua@nankai.edu.cn}
	\thanks{ Sun was supported by the National Natural Science Foundation of China (No.11501303).}
	

	\maketitle
	
	\section{Introduction}
	
	In this paper, we are mainly concerned with the existence and nonexistence of positives solutions to the following semilinear elliptic inequalities
	\begin{equation}\label{ieq}
		\Delta u+u^p\left|\nabla u\right|^q\leq0,\quad\mbox{on $G$},
	\end{equation}
	where $(p, q)\in \mathbb{R}^2$, $G=(V, E)$ is an infinite connected locally finite graph, and $V$
	is the collection of the vertices, and $E$ is the collection of edges. Throughout the paper,
	there exists only at most one edge for any two
	distinct vertices, and exists no any edge from a vertex to itself.
	
	If there is an edge connecting $x, y\in V$, we say
	$x\sim y$. On each edge, let us define an edge weight $\mu: E \to(0,\infty)$, which satisfies
	$\mu_{xy} = \mu_{yx}$, and $\mu_{xy} > 0$ if $x\sim y$. Here $\mu$ can be understood as a map from
	$V\times V\to[0,\infty)$ by adding that $\mu_{xy}=0$ provided that there is no edge connecting $x, y$.
	Such graph $G=(V, E, \mu)$ with edge weight $\mu$ is called a weighted graph. Sometimes, we use $(V, \mu)$ to denote
	the weighted graph $G$ for brevity.
	
	In this paper, we say condition $(p_0)$ is is satisfied on $G$: if there exists a constant $p_0\geq 1$ such that for any $x\sim y$ in $V$,
\begin{equation}
	\frac{\mu_{xy}}{\mu(x)}\geq \frac{1}{p_0}.\tag{$p_0$}
\end{equation}

	For each vertex  $x\in V$, let us define vertex measure $\mu(x)=\sum\limits_{x \sim y} \mu_{xy}$,
	and the Laplace operator $\Delta$ on $G$ (see \cite{G2}) as
	\begin{equation}\label{lap}
	\Delta u(x)=\sum\limits_{y \sim x} \frac{\mu_{xy}}{\mu(x)}(u(y)-u(x)),\quad\mbox{for $u\in \mathcal{l}(V)$},
	\end{equation}
	and define the gradient form $\Gamma$ (see \cite{BHLLMY}.) as
	$$\Gamma(f,g)=\sum\limits_{y \sim x}\frac{\mu_{xy}}{2\mu(x)}(f(y)-f(x))(g(y)-g(x)), \quad\mbox{for $f,g\in \mathcal{l}(V)$},$$
then the norm of gradient is defined by
	\begin{align}\label{gra}
		|\nabla u(x)|=\sqrt{\Gamma(u,u)}=\sqrt{\sum\limits_{y \sim x}\frac{\mu_{xy}}{2\mu(x)}(u(y)-u(x))^2},
		\quad\mbox{for $u\in \mathcal{l}(V)$},
	\end{align}
	where $\mathcal{l}(V)$ is the collection
	of all real functions on $V$.
	
	Besides, there exists graph distance $d(x,y)$ on $G$ which means the minimal
	number of edges in a path among all possible paths connecting $x$ and $y$ in $V$. Fix some referenced vertex $o\in V$, and for integer $n\geq1$, let
	$$B(o,n):=\{x \in V: d(o,x) \leq n\}$$
	be the closed ball centered at $o$ with radius $n$, and the volume of $B(o, n)$ be
	$$\mu(B(o, n))=\sum_{x\in B(o,n)}\mu(x).$$
	
Recently, the study on elliptic equation on weighted graphs has attracted a lot of attentions, see \cite{CM} \cite{GHJ}, \cite{GLY2} \cite{GLY3}
\cite{HSZ}, \cite{HWY} \cite{LiuY}.
	In this paper we would like to solve the following problem: what kind of sharp
	assumptions on $\mu(B(o,n))$ can suffice the nonexistence of  nontrivial positive solution $u$ to (\ref{ieq})?
	There are two folds in this problem: the first one is to find these volume growth and to prove the nonexistence
	results; the second is to show these volume assumptions are sharp.

	By using the volume assumption to obtain the nonexistence and existence of solution to elliptic differential inequalities
	is widely used in the literature. Recall the famous Nash-Williams' test (e.g. \cite{W}): if
	\begin{eqnarray}\label{votest}
		\sum^{\infty}_{n=1}\frac{n}{\mu(B(o,n))}=\infty,
	\end{eqnarray}
	then any nonnegative solution on $(V, \mu)$ is identically equal to a constant, or equivalently to say, $(V, \mu)$ is parabolic
	or recurrent.
	
	
	The notion of parabolicity of graph can be regarded as a generalization of the parabolicity of manifolds,
	see Cheng-Yau's paper \cite{CY}, and  Grigor'yan \cite{G85}, Karp\cite{K}, Varopoulos \cite{V} for further developments.

	Recently, Gu, Sun, and Huang in \cite{GSH} proved that, for $p>1$, if condition $(p_0)$ is satisfied, and if
	\begin{eqnarray}
		\mu(B(o,n))\lesssim n^{\frac{2p}{p-1}}(\ln n)^{\frac{1}{p-1}},
	\end{eqnarray}
	then $\Delta u+u^p\leq0$ admits no positive solution. While for $0<p\leq1$, $\Delta u+u^p\leq0$ admits no positive solution.
	Here the exponent $\frac{2p}{p-1}$ and $\frac{1}{p-1}$ are also showed to be sharp when $p>1$.
	
Gu-Sun-Huang's result can be considered as a discrete version obtained by Grigor'yan-Sun on manifold case in \cite{GS1}, where they proved that
\begin{eqnarray*}
\mu(B(o,r))\lesssim r^{\frac{2p}{p-1}}(\ln r)^{\frac{1}{p-1}}, \quad\mbox{for all large enough $r$},
\end{eqnarray*}
then there exists no nonnegative solution to $\Delta u+u^p\leq0$ on geodesically complete noncompact manifolds.	Here $\mu$ is Riemannian measure
on manifolds.

	Motivated by these results, we would like to study problem (\ref{ieq}) involving
	gradient terms on weighted graphs. To express our classification more clearly, let us divide $R^2$
	into six parts (Figure \ref{fig1}).
	\begin{align*}
		&G_1=\{(p,q)|p\geq 0,1-p<q<2\},\quad G_2=\{(p,q)|q\geq 2\} \\
		&G_3=\{(p,q)|p<0,1<q<2\}, \qquad \quad G_4=\{(p,q)|p<0,q=1 \}\\
		&G_5=\{(p,q)|p+q=1,p\geq0,q>0,\;\text{or}\;p+q=1,q<0\}\\
		&G_6=\{(p,q)|p<1-q,q<1,\;\text{or}\; (p,q)=(1,0)\}\\
	\end{align*}

	\begin{figure}[h]
		\begin{tikzpicture}[x={(0.8cm,0cm)},y={(0cm,0.8cm)}]\label{fig1}
			\draw[->] (-6,0)--(6,0) ;
			\draw[->] (0,-4)--(0,6);
			\fill[green,opacity=0.7] (-6,6)--(6,6)--(6,4)--(-6,4) ;
			\fill[blue,opacity=0.7] (0,4)--(6,4)--(6,-4)--(0,2) ;
			\fill[yellow,opacity=0.7] (0,4)--(-6,4)--(-6,2)--(0,2);
			\fill[red,opacity=0.7] (0,2)--(-6,2)--(-6,-4)--(6,-4);
			\draw[very thick,dashed]  (-6,2)--(0,2);
			\draw[very thick,dotted]  (6,-4)--(2,0) (2,0)--(0,2);
			\fill[blue,opacity=0.7] (7.7,5.4) rectangle(8.3,4.8);
			\fill[green,opacity=0.7] (7.7,3.6) rectangle(8.3,3);
			\fill[yellow,opacity=0.7] (7.7,1.8) rectangle(8.3,1.2);
			\fill[red,opacity=0.7] (7.7,0) rectangle(8.3,-0.6);
			\fill[red,opacity=1] (2,0) circle(0.05);
			\draw[dashed]  (7.5,-1.9)--(8.5,-1.9);
			\draw[dotted]  (7.5,-3.3)--(8.5,-3.3);
			\node[above] at (0,6) {$q$};
			\node[right] at (6,0) {$p$};
			\node[below] at (-5,4) {\tiny{$q=2$}};
			\node[below] at (-5,2) {\tiny{$q=1$}};
			\node[below] at (1.7,0) {\tiny{$(1,0)$}};
			\node[below] at (3,-1.6) {\tiny{$p+q=1$}};
			\node[right] at (8.3,5.1) {\small{$G_1$}};
			\node[right] at (8.3,3.3) {\small{$G_2$}};
			\node[right] at (8.3,1.6) {\small{$G_3$}};
			\node[right] at (8.3,-0.3) {\small{$G_6$}};
			\node[right] at (8.5,-1.9) {\small{$G_4$}};
			\node[right] at (8.5,-3.3) {\small{$G_5$}};
		\end{tikzpicture}
	\caption{}
	\end{figure}

	Our main results are as follows:
	\begin{theorem}\label{thm1} \rm{
			Let $G=(V,E, \mu)$ be an infinite, connected, locally finite graph on which  condition $(p_0)$
			is satisfied. Fix some $o\in V$.
			\begin{enumerate}
				\item[(I).]{Assume $(p,q)\in G_1$. If
					\begin{align}\label{vol-1}
						\mu(B(o,n)) \lesssim n^{\frac{2p+q}{p+q-1}}(\ln{n})^{\frac{1}{p+q-1}},
						\quad\mbox{for all $n>>1$},
					\end{align}
					then (\ref{ieq}) admits no nontrivial positive solution.}
				
				\item[(II).]{Assume $(p,q)\in G_2$. If
					\begin{align}\label{vol-2}
						\mu(B(o,n)) \lesssim n^{2}\ln{n},\quad \mbox{for all $n>>1$},
					\end{align}
					then (\ref{ieq}) admits no nontrivial positive solution.}
				
				\item[(III).]{Assume $(p,q)\in G_3$. If
					\begin{align}\label{vol-3}
						\mu(B(o,n)) \lesssim n^{\frac{q}{q-1}}(\ln{n})^{\frac{1}{q-1}},\quad \mbox{for all $n>>1$},
					\end{align}
					then (\ref{ieq}) admits no nontrivial positive solution.}
				
				\item[(IV).]{Assume $(p,q)\in G_4$.  For any given $\alpha>0$, if
					\begin{align}\label{vol-4}
						\mu(B(o,n)) \lesssim n^{\alpha}, \quad \mbox{for all $n>>1$},
					\end{align}
					then (\ref{ieq}) admits no nontrivial positive solution.}
				
				\item[(V).]{Assume $(p,q)\in G_5$. There exists some $k_0>0$, for any given $\kappa$ satisfying $0<\kappa<k_0$, if
					\begin{align}\label{vol-5}
						\mu(B(o,n)) \lesssim e^{\kappa n},\quad \mbox{for all $n>>1$},
					\end{align}
					then (\ref{ieq})  admits no nontrivial positive solution.}
				
		\end{enumerate}}
	\end{theorem}
	\begin{remark}
		In Theorem \ref{thm1} (II), by Nash-Williams's test, (\ref{vol-2}) condition can be relaxed to (\ref{votest}).
	\end{remark}

	\textbf{Notations.} In the above and below, the letters $C,C',C_0,C_1,c_0,c_1$... denote positive constants whose values are unimportant and may vary at different occurrences. $A\lesssim B$ means that the quotient of $A$ and $B$ is bounded from the above, $A\gtrsim B$ means that the quotient of $A$ and $B$ is bounded from below, and $A\asymp B$ means both $A\lesssim B$
	and $A\gtrsim B$ hold.
	
\vskip1ex

For $(p,q)\in G_6$, we have the following nonexistence result.
\begin{theorem}\label{thm1-1} \rm{
Let $G=(V,E, \mu)$ be an infinite, connected, locally finite graph.
Under any of the following two assumptions
\begin{enumerate}
\item[(1).]
{Assume $(p_0)$ condition is satisfied, and $p<1-q, q<0$; }
\item[(2).]{If either $p<1-q, 0\leq q<1$ or $(p,q)=(1,0)$;}
\end{enumerate}
then (\ref{ieq}) admits no nontrivial positive solution.}
\end{theorem}

\begin{remark}\rm{
Let us compare our results with the one obtained by Sun-Xiao-Xu on manifolds in \cite{SXX}.
On manifolds, $(p,q)\in\mathbb{R}^2$ are also divided into six parts:
\begin{eqnarray*}
&G_1^{\prime}=\{p\geq0,1-p<q<2\}, \quad &G_2^{\prime}=\{q\geq2\},\\
&G_3^{\prime}=\{p<0, 1<q<2\},\quad & G_4^{\prime}=\{p<0, q=1\},\\
&G_5^{\prime}=\{p=1-q, q\leq1\},\quad &G_6^{\prime}=\{p<1-q, q<1\}.
\end{eqnarray*}
Clearly, $G_i=G_i^{\prime}$ for $i=1,2,3,4$, $G_5=G_5^{\prime}\setminus(1,0)$, $G_6=G_6^{\prime}\cup (1, 0)$.
When  $(p, q)\in G_i$ for $i=1,2,3,4,5$,
the graph behaves quite similarly as manifolds, namely, under the volume growth of
the same forms, the nonexistence results hold. However, when $(p,q)\in G_6^{\prime}$, there exists some sharp exponentially volume growth, under which the nonexistence results hold, while on weighted graphs, there exists no such sharp volume growth to suffice nonexistence results.
Hence, it is very interesting to point out that the results on weighted graphs does not need to be accordance with these obtained on manifolds.

}
\end{remark}	
\vskip1ex

	We also show that our volume growth conditions in Theorem \ref{thm1} are sharp, roughly speaking, if the volume conditions in Theorem \ref{thm1} are relaxed, then there exist a graph which admits a nontrivial positive solution to (\ref{ieq}). More precisely, we show these sharpness on the so-called homogeneous tree $T_N$.
	
	Fix $N\geq2$, recall that a connected
	graph $(V, E)$ is a tree if for any two distinct points $x, y \in V $, there is only one path between x and y. A homogeneous tree $T_N$ is a tree whose all vertices share the same degree $N$. The following existence result is obtained on $T_N$.
	\begin{theorem}\label{thm2} \rm{
			Let $(V,E)=T_N$.
			\begin{enumerate}
				\item[(I).]{Assume $(p,q)\in G_1$. For any arbitrary small $\epsilon>0$, there exists a weight $\mu$ on $T_N$ satisfying
					\begin{align}\label{e-vol-1}
						\mu(B(o,n)) \asymp n^{\frac{2p+q}{p+q-1}}(\ln{n})^{\frac{1}{p+q-1}+\epsilon},
						\quad\mbox{for $n\geq 2$,}
					\end{align}
					such that (\ref{ieq}) admits a nontrivial positive solution on $(V, \mu)$.}
				
				\item[(II).]{Assume $(p,q)\in G_2$. For any arbitrary small $\epsilon>0$, there exists a weight $\mu$ on $T_N$ satisfying
					\begin{align}\label{e-vol-2}
						\mu(B(o,n)) \asymp n^{2}(\ln{n})^{1+\epsilon},\quad \mbox{for $n\geq 2$,}
					\end{align}
					such that (\ref{ieq}) admits a nontrivial positive solution on $(V, \mu)$.}
				
				\item[(III).]{Assume $(p,q)\in G_3$. For any arbitrary small $\epsilon>0$, there exists a weight $\mu$ on $T_N$ satisfying
					\begin{align}\label{e-vol-3}
						\mu(B(o,n)) \asymp n^{\frac{q}{q-1}}(\ln{n})^{\frac{1}{q-1}+\epsilon},\quad \mbox{for $n\geq 2$,}
					\end{align}
					such that (\ref{ieq}) admits a nontrivial positive solution on $(V, \mu)$.}
				
				\item[(IV).]{Assume $(p,q)\in G_4$.  Given $\lambda>0$, there exists a weight $\mu$ on $T_N$ satisfying
					\begin{align}\label{e-vol-4}
						\mu(B(o,n)) \asymp e^{\lambda n}, \quad \mbox{for $n\geq 2$,}
					\end{align}
					such that (\ref{ieq}) admits a nontrivial positive solution on $(V, \mu)$,}
				
				\item[(V).]{Assume $(p,q)\in G_5$. Then there exist a weight $\mu$ on $T_N$, and a positive constant $\lambda$ satisfying
					\begin{align}\label{e-vol-5}
						\mu(B(o,n)) \asymp e^{\lambda n}\qquad \mbox{for $n\geq 2$,}
					\end{align}
					such that (\ref{ieq}) admits a nontrivial positive solution on $(V, \mu)$.}
		\end{enumerate}}
	\end{theorem}
	
	\section{Proof of Theorem \ref{thm1} and \ref{thm1-1}}
	Before proceeding to the proof of Theorem \ref{thm1}, we first introduce Lemma \ref{lem1}
and \ref{lem2}, which play important roles in proof of Theorem \ref{thm1}.
		\begin{lemma}\label{lem1}\rm{
			Let $(V,\mu)$ satisfies $(p_0)$ condition. If $u$ is a nonnegative solution to (\ref{ieq}), then either $u\equiv 0$ or $u>0$ and
			\begin{align}\label{ep}
				\frac{1}{p_0}\leq\frac{u(x)}{u(y)} \leq p_0, \quad\mbox{if $y\sim x$}.
		\end{align}}
	\end{lemma}
	\begin{proof}
		The proof is similar to \cite[Lemma 3.1]{GSH}.
		%
	\end{proof}
In the following, for brevity, we denote by
 $$\nabla_{xy}f=f(y)-f(x),\quad \mbox{for $f\in \mathcal{l}(V)$}.$$

Let $\Omega$ be a non-empty subset of $V$, we say that $u$ satisfies
\begin{align}\label{vieq}
	\Delta u(x)
	+ u(x)^{p} |\nabla u(x)|^{q}\leq 0,\quad\mbox{ when $x\in \Omega $}.\end{align}
means that (\ref{vieq}) holds only for these vertices $x\in \Omega$, where $\Delta u$ and $|\nabla u|$ are still defined by
(\ref{lap}) and (\ref{gra}) respectively in $V$.
%
\begin{lemma}\label{lem2}\rm{
		Assume $p+q\neq 1$,
 $(V,\mu)$ satisfies $(p_0)$ condition,
 and $\Omega$ is a non-empty subset of $V$.
		Let $u$ be a nontrivial positive function on $V$ which satisfies (\ref{vieq}), and  $\frac{1}{p_0}\leq\frac{u(x)}{u(y)} \leq p_0$ for any $y\sim x$. Furthermore, when $\Omega\not=V$, assume $u$ also satisfies $u(y)-u(x)\geq 0$ for any $(x,y) \in\{(x,y)|y\sim x, x\in\Omega\mbox{ and }y\in \Omega^c\}$.
		Then there exists a positive pair
		$(s,t)$ such that for any $0\leq\varphi\leq 1$ with compact support in $\Omega $, the following estimates hold:
		\begin{align}\label{est-1}
			&\sum\limits_{x \in \Omega }\mu(x)  u(x)^{p-t}|\nabla u(x)|^q\varphi(x)^s \nonumber\\
			\leq& C_{p_0,t} (2s)^{\frac{2p+q+t(q-2)}{p+q-1}} t^{- \frac{p+t(q-1)}{p+q-1}}
			\left(\sum_{\substack {x,y\in \Omega \\ \nabla_{xy} \varphi \neq 0}} \mu_{xy}\varphi(x)^s u(x)^{p-t}|\nabla u(x)|^q\right)^{\frac{1-t}{p+q-t}}\nonumber\\
			&\times\left(\sum\limits_{x,y\in \Omega } \mu_{xy} |\nabla_{xy}\varphi|^{\frac{2p+q+t(q-2)}{p+q-1}}\right)^{\frac{p+q-1}{p+q-t}},
		\end{align}
		and
		\begin{align}\label{est-2}
			\sum\limits_{x \in \Omega }\mu(x)  u(x)^{p-t} |\nabla u(x)|^q \varphi(x)^s
			\leq& ( C'_{p_0,t})^{\frac{p+q-t}{p+q-1}}  (2s)^{\frac{2p+q+t(q-2)}{p+q-1}} t^{-\frac{p+t(q-1)}{p+q-1}}\nonumber\\
			&\times\sum\limits_{x,y\in \Omega } \mu_{xy} |\nabla_{xy}\varphi|^{\frac{2p+q+t(q-2)}{p+q-1}},
		\end{align}
		where
		$C_{p_0,t}= \frac{(\sqrt{2p_0}(1+p^t_0))^{\frac{p + t(q-1)}{p + q-t}+1}   (p_0^{t+1})^{\frac{p + t(q-1)}{p + q-t}}}{4}$, $C'_{p_0,t} =( C_{p_0,t})^{\frac{p+q-t}{p+q-1}}$,
		and $s,t$ satisfy
		\begin{equation}\label{st-cond}
			\left\{
			\begin{array}{lr}
				\frac{2p + q + t(q-2)}{p + q-t} > 1, \\
				\frac{p+q-t}{1-t} > 1,\\
				s > \frac{2p+q+t(q-2)}{p+q-1}.
			\end{array}
			\right.
	\end{equation}}
\end{lemma}

\begin{proof}
	For  $ \varphi \in \mathcal{l}(\Omega) $ with compact support in $\Omega $, define $\psi=\varphi^s u^{-t}$, where $(s,t)$ are to be chosen later.
	
	Multiplying both sides of (\ref{vieq}) by $\mu(x)\psi(x)$ and summing up over all $x\in \Omega $, we obtain
	\begin{equation*}
		\sum\limits_{x \in \Omega,y\in V }\mu_{xy}(\nabla_{xy}u)\psi(x)
		+\sum\limits_{x \in \Omega }\mu(x)  u(x)^{p} |\nabla u(x)|^{q}\psi(x)\leq 0,
	\end{equation*}
It follows that
\begin{align}\label{lem2-1}
	&\sum\limits_{x,y \in \Omega}\mu_{xy}(\nabla_{xy}u)\psi(x)+\sum\limits_{x \in \Omega, y\in \Omega^c}\mu_{xy}(\nabla_{xy}u)\psi(x)\nonumber \\
	&+\sum\limits_{x \in \Omega }\mu(x)  u(x)^{p} |\nabla u(x)|^{q}\psi(x)\leq 0.
\end{align}
Specially, when $\Omega=V$, we have $\sum\limits_{x \in \Omega, y\in \Omega^c}\mu_{xy}(\nabla_{xy}u)\psi(x) =0$.

	Noting
	\begin{equation*}
		\sum\limits_{x,y \in \Omega }\mu_{xy}(\nabla_{xy}u)\psi(x)=
		-\frac{1}{2}\sum\limits_{x,y \in \Omega }\mu_{xy}(\nabla_{xy}u)(\nabla_{xy}\psi),
	\end{equation*}
	and
	\begin{equation*}
		\nabla_{xy}\psi=\nabla_{xy}(\varphi^s u^{-t})=
		u(y)^{-t} \nabla_{xy}(\varphi^s)+ \varphi(x)^s \nabla_{xy}(u^{-t}),
	\end{equation*}
	we obtain
	\begin{align*}
			\sum\limits_{x,y \in \Omega }\mu_{xy}(\nabla_{xy}u)\psi(x)=&
			-\frac{1}{2}\sum\limits_{x,y \in \Omega }\mu_{xy} u(y)^{-t} (\nabla_{xy}u) \nabla_{xy}(\varphi^s)\nonumber\\
			&-\frac{1}{2}\sum\limits_{x,y \in \Omega}\mu_{xy} \varphi(x)^s (\nabla_{xy}u)
			\nabla_{xy}(u^{-t}).
	\end{align*}
	Then (\ref{lem2-1}) is transformed to
	\begin{align}\label{lem2-2}
		&-\frac{1}{2}\sum\limits_{x,y \in \Omega }\mu_{xy} \varphi(x)^s (\nabla_{xy}u)
		\nabla_{xy}(u^{-t})+\sum\limits_{x \in \Omega, y\in \Omega^c}\mu_{xy}(\nabla_{xy}u)\varphi(x)^su(x)^{-t}\nonumber \\
		& +\sum\limits_{x \in \Omega }\mu(x)  u(x)^{p-t} |\nabla u(x)|^{q} \varphi(x)^s
		\leq \frac{1}{2}\sum\limits_{x,y \in \Omega }\mu_{xy} u(y)^{-t} (\nabla_{xy}u) \nabla_{xy}(\varphi^s).
	\end{align}
Using the mid-value theorem, we have some $\xi$ which is between $u(y)$ and $u(x)$, such that
	\begin{equation*}
		\nabla_{xy}(u^{-t})=u(y)^t-u(x)^t=-t\xi^{-t-1}(u(y)-u(x))
		=-t\xi^{-t-1}\nabla_{xy}u,
	\end{equation*}
	By $\frac{1}{p_0}\leq\frac{u(x)}{u(y)} \leq p_0$,
	we have
	$\frac{u(x)}{p_0} \leq \xi \leq u(x) p_0$, and
	\begin{align}\label{lem2-3-1}
		&-\frac{1}{2}\sum\limits_{x,y \in \Omega }\mu_{xy} \varphi(x)^s (\nabla_{xy}u)
		\nabla_{xy}(u^{-t})\nonumber\\
		&=\frac{t}{2}\sum\limits_{x,y \in \Omega }\mu_{xy} \varphi(x)^s (\nabla_{xy}u)^2
		\xi^{-t-1}\nonumber\\
		&\geq \frac{t}{2p_0^{t+1}}\sum\limits_{x,y \in \Omega }\mu_{xy} \varphi(x)^su(x)^{-t-1} (\nabla_{xy}u)^2,
	\end{align}
By $0<\frac{u(y)-u(x)}{u(x)}\leq \frac{p_0u(x)-u(x)}{u(x)}=p_0-1$, we obtain
\begin{align}\label{lem2-3-2}
	&\sum\limits_{x \in \Omega, y\in \Omega^c}\mu_{xy}(\nabla_{xy}u)\varphi(x)^su(x)^{-t}
	\nonumber\\&
	\geq \frac{1}{p_0-1}\sum\limits_{x \in \Omega, y\in \Omega^c}\mu_{xy}(\nabla_{xy}u)^2\varphi(x)^su(x)^{-t-1}
	\nonumber\\&
	\geq \frac{t}{2p_0^{t+1}}\sum\limits_{x \in \Omega, y\in \Omega^c}\mu_{xy}(\nabla_{xy}u)^2\varphi(x)^su(x)^{-t-1}.
\end{align}
where we have used that if $\Omega\neq V$, $u(y)>u(x)$ for $x\in\Omega$, $y\in\Omega^c$, and $2p_0^{t+1}\geq t(p_0-1)$ holds for all $t\geq0$.
	
	Combining (\ref{lem2-3-1}) with (\ref{lem2-3-2}), we get
	\begin{align}\label{lem2-3}
		-\frac{1}{2}\sum\limits_{x,y \in \Omega }&\mu_{xy} \varphi(x)^s (\nabla_{xy}u)
		\nabla_{xy}(u^{-t})+\sum\limits_{x \in \Omega, y\in \Omega^c}\mu_{xy}(\nabla_{xy}u)\varphi(x)^su(x)^{-t}
		\nonumber\\&
		\geq  \frac{t}{2p_0^{t+1}}\sum\limits_{x \in \Omega, y\in V}\mu_{xy}(\nabla_{xy}u)^2\varphi(x)^su(x)^{-t-1}
		\nonumber\\&
		=\frac{t}{p_0^{t+1}}\sum\limits_{x \in \Omega}\mu(x)|\nabla u(x)|^2\varphi(x)^su(x)^{-t-1}.
	\end{align}
Especially, when $\Omega=V$, (\ref{lem2-3}) can be deduced from (\ref{lem2-3-1}) directly.

	By the mid-value theorem, there is some $\eta$ between $\varphi(x)$ and $\varphi(y)$ such that
	\begin{equation}\label{lem2-4}
		\nabla_{xy}(\varphi^s)=s\eta^{s-1}(\varphi(y)-\varphi(x))
		=s\eta^{s-1}\nabla_{xy}\varphi.
	\end{equation}
	Substituting (\ref{lem2-4}) and (\ref{lem2-3}) into (\ref{lem2-2}), we have
	\begin{align}\label{lem2-5}
		&\frac{t}{p_0^{t+1}}\sum\limits_{x \in \Omega }\mu(x) \varphi(x)^s      u(x)^{-t-1} |\nabla u(x)|^2
		+\sum\limits_{x \in \Omega }\mu(x)  u(x)^{p-t} |\nabla u(x)|^{q} \varphi(x)^s\nonumber\\&
		\leq \frac{s}{2}\sum\limits_{x,y \in \Omega }\mu_{xy} u(y)^{-t}\eta^{s-1} (\nabla_{xy}u)( \nabla_{xy}\varphi).
	\end{align}
	Observing that $|\nabla u(x)|^2=\sum\limits_{y \in V }\frac{\mu_{xy}}{2\mu(x)}(\nabla_{xy}u)^2$, and
	$ \frac{1}{2p_0} \leq \frac{\mu_{xy}}{2\mu(x)} \leq \frac{1}{2}$, we derive
	\begin{align}\label{grd}
		|\nabla_{xy}u| \leq \sqrt{2p_0}|\nabla u(x)|,\quad \mbox{for any $y\sim x$.}
	\end{align}
	
	 Since $\eta^{s-1} \leq\varphi(x)^{s-1}+\varphi(y)^{s-1} $, $\frac{u(x)}{p_0} \leq \xi \leq u(x) p_0 $, and (\ref{grd}), we have
	\begin{align}\label{lem2-6}
		\frac{s}{2}&\sum\limits_{x,y \in \Omega }\mu_{xy} u(y)^{-t}\eta^{s-1} (\nabla_{xy}u)( \nabla_{xy}\varphi)\nonumber \\
		&\leq \frac{s}{2}\sum\limits_{x,y \in \Omega }\mu_{xy} u(y)^{-t}(\varphi(x)^{s-1}+\varphi(y)^{s-1}) (\nabla_{xy}u)( \nabla_{xy}\varphi)\nonumber \\
		&\leq \frac{s}{2}(1+p^t_0)\sum\limits_{x,y \in \Omega }\mu_{xy} u(x)^{-t}\varphi(x)^{s-1} (\nabla_{xy}u)( \nabla_{xy}\varphi)\nonumber \\
		&\leq \frac{s}{2}(1+p^t_0)\sum\limits_{x,y \in \Omega }\mu_{xy} u(x)^{-t}\varphi(x)^{s-1} |\nabla_{xy}u|| \nabla_{xy}\varphi|\nonumber\\
		&\leq\frac{s}{2}\sqrt{2p_0} (1+p^t_0)\sum\limits_{x,y \in \Omega }\mu_{xy} u(x)^{-t}\varphi(x)^{s-1} |\nabla u(x)||\nabla_{xy}\varphi|.
	\end{align}
	Let
	\begin{equation}\label{def-ab}
		a=\frac{2p + q + t(q-2)}{p + q-t},\quad b=\frac{2p + q + t(q-2)}{p + t(q-1)},
	\end{equation}
	and $t$ to be chosen later such that $a, b\geq 1$.
	
	By applying Young's inequality, we obtain
	\begin{align}\label{lem2-7}
		\frac{s}{2}&\sum\limits_{x,y \in \Omega }\mu_{xy} u(x)^{-t}\varphi(x)^{s-1} |\nabla u(x)||\nabla_{xy}\varphi|
		\nonumber \\
		=&\frac{1}{2}\sum\limits_{x,y \in \Omega }\left(\mu_{xy}^{\frac{1}{b}}(\frac{t}{2})^{\frac{1}{b}}
		u(x)^{-\frac{t+1}{b}} |\nabla u(x)|^{\frac{2}{b}} \varphi(x)^{\frac{s}{b}}\right)\nonumber \\
		&\quad\quad\times \left(\mu_{xy}^{\frac{1}{a}}s(\frac{t}{2})^{-\frac{1}{b}}
		u(x)^{-t+\frac{t+1}{b}} |\nabla u(x)|^{1-\frac{2}{b}} \varphi(x)^{s-1-\frac{s}{b}} |\nabla_{xy}\varphi|\right) \nonumber\\
		\leq& \frac{\epsilon t}{4}\sum\limits_{x,y \in \Omega } \mu_{xy} u(x)^{-t-1}  |\nabla u(x)|^2 \varphi(x)^s\nonumber \\
		&+\epsilon ^{-\frac{a}{b}}s^a 2^{a-1} t^{1-a}\sum\limits_{x,y \in \Omega } \mu_{xy} u(x)^{-t+a-1} |\nabla u(x)|^{2-a} \varphi(x)^{s-a}|\nabla_{xy}\varphi|^a\nonumber \\
		\leq & \frac{\epsilon t}{4}\sum\limits_{x \in \Omega } \mu(x) u(x)^{-t-1}  |\nabla u(x)|^2 \varphi(x)^s\nonumber \\
		&+\epsilon ^{-\frac{a}{b}}s^a 2^{a-1} t^{1-a}\sum\limits_{x,y \in \Omega } \mu_{xy} u(x)^{-t+a-1} |\nabla u(x)|^{2-a} \varphi(x)^{s-a}|\nabla_{xy}\varphi|^a.
	\end{align}
where we used $\sum\limits_{y \in \Omega} \mu_{xy}\leq \mu(x)$, for any $x\in \Omega$.

	Letting $\epsilon =\frac{2}{\sqrt{2p_0} (1+p^t_0) p_0^{t+1}}$, and substituting (\ref{lem2-7}) into (\ref{lem2-6}), we obtain
	\begin{align}\label{lem2-8}
		\frac{s}{2}&\sum\limits_{x,y \in V}\mu_{xy} u(y)^{-t}\eta^{s-1} (\nabla_{xy}u)( \nabla_{xy}\varphi)\nonumber \\
		\leq& \frac{t}{2p_0^{t+1}}\sum\limits_{x \in \Omega } \mu(x) u(x)^{-t-1}  |\nabla u(x)|^2 \varphi(x)^s\nonumber \\
		&+C_{p_0,t}(2s)^a  t^{1-a}\sum\limits_{x,y \in \Omega } \mu_{xy} u(x)^{-t+a-1} |\nabla u(x)|^{2-a} \varphi(x)^{s-a} |\nabla_{xy}\varphi|^a,
	\end{align}
	where
	$$C_{p_0,t}= \frac{(\sqrt{2p_0}(1+p^t_0))^{\frac{p + t(q-1)}{p + q-t}+1}   (p_0^{t+1})^{\frac{p + t(q-1)}{p + q-t}}}{4}.$$
	Combining (\ref{lem2-8}) with (\ref{lem2-5}), we have
	\begin{align*}
		&\frac{t}{2p_0^{t+1}} \sum\limits_{x \in \Omega }\mu(x)\varphi(x)^s u(x)^{-t-1} |\nabla u(x)|^2
		+\sum\limits_{x \in \Omega }\mu(x)  u(x)^{p-t} |\nabla u(x)|^{q} \varphi(x)^s \\
		&
		\leq C_{p_0,t}(2s)^a  t^{1-a}\sum\limits_{x,y \in \Omega } \mu_{xy} u(x)^{-t+a-1} |\nabla u(x)|^{2-a} \varphi(x)^{s-a} |\nabla_{xy}\varphi|^a.
	\end{align*}
	It follows that
	\begin{align}\label{lem2-10}
		&\sum\limits_{x \in \Omega }\mu(x)  u(x)^{p-t} |\nabla u(x)|^{q} \varphi(x)^s \nonumber\\
		\leq& C_{p_0,t}(2s)^a  t^{1-a}\sum\limits_{x,y \in \Omega } \mu_{xy} u(x)^{-t+a-1} |\nabla u(x)|^{2-a} \varphi(x)^{s-a} |\nabla_{xy}\varphi|^a.
	\end{align}
	
	Defining $\gamma=\frac{p+q-t}{1-t},\rho=\frac{p+q-t}{p+q-1}$, and choosing $t$ to make $\gamma,\rho>1$, and applying H\"{o}lder's inequality to RHS of (\ref{lem2-10}), we obtain
	\begin{align}\label{lem2-11}
		&\sum\limits_{x,y \in \Omega } \mu_{xy} u(x)^{-t+a-1} |\nabla u(x)|^{2-a} \varphi(x)^{s-a}|\nabla_{xy}\varphi|^a  \nonumber\\&
		=\sum\limits_{x,y\in \Omega \atop \nabla_{xy} \varphi \neq 0} \mu_{xy}\left(u(x)^{-t+a-1} |\nabla u(x)|^{2-a}\varphi(x)^{\frac{s}{\gamma}}\right)\left(\varphi(x)^{s-a-\frac{s}{\gamma}}|\nabla_{xy}\varphi|^a \right) \nonumber\\&
		\leq \left(\sum\limits_{x,y\in \Omega \atop \nabla_{xy} \varphi \neq 0} \mu_{xy}u(x)^{p-t}
		|\nabla u(x)|^q \varphi(x)^s\right)^{\frac{1}{\gamma}}
		\left(\sum\limits_{x,y\in \Omega } \mu_{xy} \varphi(x)^{s-a\rho} |\nabla_{xy}\varphi|^{a\rho}\right)^{\frac{1}{\rho}},
	\end{align}

	Choosing large enough $s$ to let $ s \geq a\rho$, and noticing $0 \leq \varphi \leq 1$, we derive
	\begin{align}\label{lem2-12}
		&\sum\limits_{x,y \in \Omega } \mu_{xy} u(x)^{-t+a-1} |\nabla u(x)|^{2-a} \varphi(x)^{s-a}|\nabla_{xy}\varphi|^a  \nonumber\\&
		\leq \left(\sum\limits_{x,y\in \Omega \atop \nabla_{xy} \varphi \neq 0} \mu_{xy}u(x)^{p-t}
		|\nabla u(x)|^q \varphi(x)^s\right)^{\frac{1}{\gamma}}
		\left(\sum\limits_{x,y\in \Omega } \mu_{xy} |\nabla_{xy}\varphi|^{a\rho}\right)^{\frac{1}{\rho}}.
	\end{align}
Substituting (\ref{lem2-12}) into (\ref{lem2-10}), we get
	\begin{align*}
		&\sum\limits_{x \in \Omega }\mu(x)  u(x)^{p-t} |\nabla u(x)|^{q} \varphi(x)^s \nonumber\\&
		\leq C_{p_0,t}(2s)^a  t^{1-a} \left(\sum\limits_{x,y\in \Omega \atop \nabla_{xy} \varphi \neq 0} \mu_{xy}u(x)^{p-t}
		|\nabla u(x)|^q \varphi(x)^s\right)^{\frac{1}{\gamma}}
		\left(\sum\limits_{x,y\in \Omega } \mu_{xy} |\nabla_{xy}\varphi|^{a\rho}\right)^{\frac{1}{\rho}}.
	\end{align*}
	Combining the above with (\ref{def-ab}), we derive
	\begin{align*}
		&\sum\limits_{x \in \Omega }\mu(x)  u(x)^{p-t} |\nabla u(x)|^{q} \varphi(x)^s \nonumber\\
		\leq& C_{p_0,t}(2s)^{\frac{2p + q + t(q-2)}{p + q-t}}  t^{-\frac{p + t(q-1)}{p + q-t}}\left(\sum\limits_{x,y\in \Omega \atop \nabla_{xy} \varphi \neq 0} \mu_{xy}u(x)^{p-t}
		|\nabla u(x)|^q \varphi(x)^s\right)^{\frac{1-t}{p+q-t}} \nonumber\\&
		\times \left(\sum\limits_{x,y\in \Omega } \mu_{xy}  |\nabla_{xy}\varphi|^{\frac{2p+q+t(q-2)}{p+q-1}}\right)^{\frac{p+q-1}{p+q-t}}.
	\end{align*}
	then (\ref{est-1}) follows.
	
	Noting $\sum\limits_{x \in \Omega }\mu(x)u(x)^{p-t} |\nabla u(x)|^{q} \varphi(x)^s$ is finite and
	\begin{align*}
		&\sum\limits_{x,y\in \Omega \atop \nabla_{xy} \varphi \neq 0} \mu_{xy}u(x)^{p-t}
		|\nabla u(x)|^q \varphi(x)^s\leq \sum\limits_{x\in \Omega,y\in V} \mu_{xy}u(x)^{p-t}
		|\nabla u(x)|^q \varphi(x)^s\\&
		=\sum\limits_{x \in \Omega }\mu(x)u(x)^{p-t} |\nabla u(x)|^{q} \varphi(x)^s,
	\end{align*}
we obtain
	\begin{align*}
&\left( \sum\limits_{x \in \Omega }\mu(x)  u(x)^{p-t} |\nabla u(x)|^{q} \varphi(x)^s \right)^{\frac{p+q-1}{p+q-t}} \nonumber\\&
\leq C_{p_0,t}(2s)^{\frac{2p + q + t(q-2)}{p + q-t}}  t^{-\frac{p + t(q-1)}{p + q-t}} \left(\sum\limits_{x,y\in \Omega } \mu_{xy}  |\nabla_{xy}\varphi|^{\frac{2p+q+t(q-2)}{p+q-1}}\right)^{\frac{p+q-1}{p+q-t}}.
\end{align*}
Hence we have
	\begin{align*}
		\sum\limits_{x \in \Omega }\mu(x)  u(x)^{p-t} |\nabla u(x)|^{q} \varphi(x)^s
		\leq&( C_{p_0,t})^{\frac{p+q-t}{p+q-1}} (2s)^{\frac{2p+q+t(q-2)}{p+q-1}} t^{-\frac{p+t(q-1)}{p+q-1}}\nonumber\\
		&\times
		\sum\limits_{x,y\in \Omega } \mu_{xy} |\nabla_{xy}\varphi|^{\frac{2p+q+t(q-2)}{p+q-1}},
	\end{align*}
	which implies (\ref{est-2}). Hence, we complete the proof.
\end{proof}

\begin{remark}\label{rem}
	In Lemma \ref{lem2}, since $s$ is only needed to be chosen large enough, it
	suffices to verify that such $t$ exists. For our convenience, let us divide $R^2\setminus \{p+q=1\}$ into four different parts $K_1$, $K_2$, $K_3$, $K_4$ (see figure \ref{fig2})
	\begin{align*}
		&K_1=\{(p,q)|p<1-q, q\leq1\},\quad K_2=\{(p,q)|p\geq0, 1-p<q\leq 1\} \\
		&K_3=\{(p,q)|p>1-q, q>1\}, \quad K_4=\{(p,q)|p<0,1<q<1-p\}
	\end{align*}
		\begin{figure}[H]
	\begin{tikzpicture}[x={(0.8cm,0cm)},y={(0cm,0.8cm)}]
		\draw[->] (-6,0)--(4,0) ;
		\draw[->] (0,-2)--(0,6);
		\fill[green,opacity=0.7] (0,2)--(4,2)--(4,6)--(-4,6) ;
		\fill[blue,opacity=0.7] (0,2)--(4,2)--(4,-2);
		\fill[yellow,opacity=0.7] (0,2)--(-6,2)--(-6,6)--(-4,6);
		\fill[red,opacity=0.7] (0,2)--(-6,2)--(-6,-2)--(4,-2);
		\fill[red,opacity=0.7] (5.7,5.4) rectangle(6.3,4.8);
		\fill[blue,opacity=0.7] (5.7,3.6) rectangle(6.3,3);
		\fill[green,opacity=0.7] (5.7,1.8) rectangle(6.3,1.2);
		\fill[yellow,opacity=0.7] (5.7,0) rectangle(6.3,-0.6);
		\node[above] at (0,6) {$q$};
		\node[right] at (4,0) {$p$};
		\node[below] at (-5,2) {\tiny{$q=1$}};
		\node[below] at (-2,3.5) {\tiny{$p+q=1$}};
		\node[right] at (6.3,5.1) {\small{$K_1$}};
		\node[right] at (6.3,3.3) {\small{$K_2$}};
		\node[right] at (6.3,1.6) {\small{$K_3$}};
		\node[right] at (6.3,-0.3) {\small{$K_4$}};
	\end{tikzpicture}
	\caption{}
	\label{fig2}
\end{figure}
According to the location of $(p, q)$, we choose $t$ in the following way
\begin{enumerate}
\item[1.]{When $(p, q)\in K_1$, take $t>1$; }
\item[2.]{When $(p, q)\in K_2$, take $0<t<1$;}
\item[3.]{When $(p, q)\in K_3$, take $\max\{-\frac{p}{q-1}, 0\}<t<1$;}
\item[4.]{When $(p, q)\in K_4$, take $1<t<-\frac{p}{q-1}$.}
\end{enumerate}	

\end{remark}
Now we step into the proof of Theorem \ref{thm1}.
\begin{proof}[\rm{Proof of Theorem \ref{thm1} (I)}]

Assume u is a nontrivial positive solution to (\ref{ieq}). By Lemma \ref{lem1}, u satisfies $\frac{1}{p_0}\leq\frac{u(x)}{u(y)} \leq p_0$ for any $y\sim x$.

	For our convenience, let us denote $d(x)= d(o,x)$, and $B_i=\{x\in V:d(x)\leq 2^i\}$ for integer $i$.
	Fix integer $n$, define $h_n$ on $V$ as
	\begin{align}\label{hn}
		h_n(x)=
		\left\{
		\begin{array}{lr}
			1, \qquad \qquad d(x)\leq n,\\
			2-\frac{d(x)}{n}, \quad n<d(x)<2n,\\
			0,  \qquad \qquad  d(x)\geq 2n,
		\end{array}
		\right.
	\end{align}
	and let
	\begin{align}\label{def-phi}
		\varphi_i=\frac{1}{i}\sum\limits^{2i-2}_{k=i-1}h_{2^k}(x).
	\end{align}
It follows that $0\leq\varphi_i\leq1$, and $\varphi_i=1$ in $B_{i-1}$, $\varphi_i=0$ in $(B_{2i-1})^c$. Moreover, for any $x\in B_k-B_{k-1}$, if $k \leq i - 2$ or
	$k \geq 2i + 1$, then $\nabla_{xy}\varphi_i = 0$ for any $y \sim x$; while if $i-1\leq k \leq 2i$, we have
	$$|\nabla_{xy}\varphi_i|\lesssim \frac{1}{i\cdot 2^k}, \quad\mbox{for any $y\sim x$}.$$
	
	Letting $\Omega=V$ in Lemma \ref{lem2}, and substituting $\varphi=\varphi_i$ in (\ref{est-2}), we get
	\begin{align}
		&\sum\limits_{x \in V}\mu(x)  u(x)^{p-t} |\nabla u(x)|^q \varphi_i(x)^s
		\nonumber \\&
		\lesssim C_{p_0,t}  (2s)^{\frac{2p+q+t(q-2)}{p+q-1}} t^{-\frac{p+t(q-1)}{p+q-1}}
		\sum\limits^{2i}_{k=i-1}
		\sum\limits_{x\in B_k-B_{k-1}}
		\sum\limits_{y \sim x} \mu_{xy} |\nabla_{xy}\varphi|^{\frac{2p+q+t(q-2)}{p+q-1}}
		\nonumber \\&
		\lesssim C_{p_0,t}  (2s)^{\frac{2p+q+t(q-2)}{p+q-1}} t^{-\frac{p+t(q-1)}{p+q-1}}
		\sum\limits^{2i}_{k=i-1}
		\sum\limits_{x\in B_k-B_{k-1}}
		\mu(x) (\frac{1}{i\cdot 2^k})^{\frac{2p+q+t(q-2)}{p+q-1}}
		\nonumber \\&
		\lesssim C_{p_0,t}  (2s)^{\frac{2p+q+t(q-2)}{p+q-1}} \frac{t^{-\frac{p+t(q-1)}{p+q-1}}}{i^{\frac{2p+q+t(q-2)}{p+q-1}}}
		\sum\limits^{2i}_{k=i-1}
		\mu(B_k) 2^{-k\frac{2p+q+t(q-2)}{p+q-1}},
	\end{align}
	where $C_{p_0,t}=\frac{(\sqrt{2p_0}(1+p^t_0))^{\frac{p + t(q-1)}{p + q-t}+1}   (p_0^{t+1})^{\frac{p + t(q-1)}{p + q-t}}}{4}$.
	
	Since $G_1\subset K_2 \cup K_3$, and by Remark \ref{rem}, we choose
	$$t=\frac{1}{i},$$
	and $s$ to be some large fixed constant.
	Recalling (\ref{vol-1}), and noting that $p+q>1$ and $q<2$, we obtain
 \begin{align}\label{I-3}
		&\sum\limits_{x \in V}\mu(x)  u(x)^{p-1/i} |\nabla u(x)|^q \varphi_i(x)^s
		\nonumber \\&
		\lesssim i^{-1-\frac{1-1/i}{p+q-1}} \sum\limits^{2i}_{k=i-1}
		2^{\frac{k(2-q)}{i(p+q-1)}} k^{\frac{1}{p+q-1}}
		\nonumber \\&
		\lesssim i^{-1+\frac{1/i}{p+q-1}} \sum\limits^{2i}_{k=i-1}
		2^{\frac{k(2-q)}{i(p+q-1)}}
		\nonumber \\&
		\lesssim i^{\frac{1/i}{p+q-1}}.
	\end{align}
Consequently from letting $i \rightarrow \infty$ in (\ref{I-3})
	$$\sum\limits_{x \in V}\mu(x)  u(x)^{p} |\nabla u(x)|^q <\infty. $$
	Substituting $\varphi=\varphi_i$ and $t=\frac{1}{i}$ into (\ref{est-1}), and repeating the same procedures, we have
	$$\lim_{i\to\infty}\sum\limits_{x \in V}\mu(x)  u(x)^{p-1/i} |\nabla u(x)|^q \varphi_i(x)^s
	= 0,$$
	namely
	$$\sum\limits_{x \in V}\mu(x)  u(x)^{p} |\nabla u(x)|^q = 0,$$
	which is a contradiction to the assumption that $u$ is nontrivial. Hence, the proof of Theorem \ref{thm1} (I) is complete.
\end{proof}
\begin{proof}[\rm\textbf{Proof of Theorem \ref{thm1} (II)}]
	Let us divide the proof into three cases:
	\begin{enumerate}
		\item[(II-1).]{$(p,q)\in \{p+q>1,q>2\};$}
		\item[(II-2).]{$(p,q)\in \{p+q=1,q>2\};$}
		\item[(II-3).]{$(p,q)\in \{p+q<1,q>2\}.$}
	\end{enumerate}
	
	In case (II-1), it follows that $(p,q)\in K_2$. Hence let
	\begin{align*}
		t=1-\frac{1}{i},
	\end{align*}
	and $s$ be some large fixed constant.
	
	Letting $\Omega=V$ in Lemma \ref{lem2}, substituting $\varphi=\varphi_i$ from (\ref{def-phi}) into (\ref{est-2}), and using the same technique as in (\ref{I-3}), we obtain
	\begin{align}\label{2-1}
		&\sum\limits_{x \in V}\mu(x)  u(x)^{p-t} |\nabla u(x)|^q \varphi_i(x)^s
		\nonumber \\&
		\lesssim C_{p_0,t}  (2s)^{\frac{2p+q+t(q-2)}{p+q-1}} \frac{t^{-\frac{p+t(q-1)}{p+q-1}}}{i^{\frac{2p+q+t(q-2)}{p+q-1}}}
		\sum\limits^{2i}_{k=i-1}
		\mu(B_k) 2^{-k\frac{2p+q+t(q-2)}{p+q-1}},
	\end{align}
	Combining with (\ref{vol-3}) and (\ref{2-1}), and noting that $ C_{p_0,t}  (2s)^{\frac{2p+q+t(q-2)}{p+q-1}}t^{-\frac{p+t(q-1)}{p+q-1}}$ is uniformly bounded
	for $ i $, we obtain
	\begin{align}\label{2-2}
		&\sum\limits_{x \in V}\mu(x)  u(x)^{p-t} |\nabla u(x)|^q \varphi_i(x)^s
		\nonumber \\&
		\lesssim  i^{-\frac{2p+q+t(q-2)}{p+q-1}}
		\sum\limits^{2i}_{k=i-1}
		\mu(B_k) 2^{-k\frac{2p+q+t(q-2)}{p+q-1}}
		\nonumber \\&
		\lesssim  i^{-\frac{2p+q+t(q-2)}{p+q-1}}
		\sum\limits^{2i}_{k=i-1} 2^{k(2-\frac{2p+q+t(q-2)}{p+q-1})}k
		\nonumber \\&
		\lesssim i^{1-\frac{2p+q+t(q-2)}{p+q-1}}
		\sum\limits^{2i}_{k=i-1} 2^{\frac{k(q-2)}{i(p+q-1)}}.
	\end{align}
	Substituting $ t=1-\frac{1}{i}$ into (\ref{2-2}), we get
	\begin{align}\label{2-3}
		&\sum\limits_{x \in V}\mu(x)  u(x)^{p-1+\frac{1}{i}} |\nabla u(x)|^q \varphi_i(x)^s
		\nonumber \\&
		\lesssim i^{1-\frac{2p+q+t(q-2)}{p+q-1}}
		\sum\limits^{2i}_{k=i-1} 2^{\frac{k(q-2)}{i(p+q-1)}}
		\nonumber \\&
		\lesssim i^{\frac{q-2}{i(p+q-1)}}.
	\end{align}
	Letting $i \rightarrow \infty$ in (\ref{2-3}), we have
	$$\sum\limits_{x \in V}\mu(x)  u(x)^{p-1} |\nabla u(x)|^q <\infty. $$
	Substituting $\varphi=\varphi_i$ and $t=1-\frac{1}{i}$ into (\ref{est-1}), and repeating the same procedures as in the proof of Theorem \ref{thm1} (I), we derive
	\begin{align}
		\sum\limits_{x \in V}\mu(x)  u(x)^{p-1} |\nabla u(x)|^q = 0.
	\end{align}
	which yields a contradiction with the nontrivialness of $u$.
	
	In case (II-2),
	denote $\Omega_k=\{ x\in V|0<u(x)<k \}$.
	Since $u$ is a nontrivial positive solution, then there exists a constant $k_0$ such that $\mu(\Omega_k)>0 $ for any $k>k_0$. Now fix such $k$, let $v=\frac{u}{k}$, and $v$ satisfies
	$$ \Delta v+v^p\left|\nabla v\right|^q\leq0, \quad\mbox{ on $V$}.$$
	Obiviously $0<v<1$ in $\Omega_k$, we have
	\begin{align}\label{2-2-1}
		\Delta v+v^{p+\epsilon}\left|\nabla v\right|^q\leq0, \quad\mbox{ on $\Omega_k$},
	\end{align}
	where $p'=p+\epsilon$ and $\epsilon>0$, thus $(p^{\prime},q)\in \{ p^{\prime}+q=p+q+\epsilon>1,q>2 \} $, consequently
	$(p^{\prime},q)\in$ (II-1).
	
	From the definition of $v(x)$ and $\Omega_k$ , we know $\frac{1}{p_0}\leq\frac{v(x)}{v(y)}=\frac{u(x)}{u(y)} \leq p_0$ and  $v(y)-v(x)\geq 0$ when $x\in \Omega_k$, $y\in \Omega_k^c$ .
Hence by Lemma \ref{lem2}, and by taking the same procedure as in case (II-1) except replacing $V$ with $\Omega_k$, we arrive
	\begin{align}\label{2-2-4}
		\sum\limits_{x \in \Omega_k}\mu(x)  v(x)^{p'-1} |\nabla v(x)|^q = 0,\quad\mbox{ on $\Omega_k$}.
	\end{align}
	Let $k_i=\max\{u(x)|d(o,x)\leq i\}+k_0$, we have
	$B(o,i) \subset \Omega_{k_i} $. Taking $k=k_i$ in (\ref{2-2-4}), we obtain that
	$v\equiv cons.$ in $B(o,i)$, which implies that $u\equiv cons.$ in $B(o,i)$.
	
	Letting $i\rightarrow \infty$, we get $u\equiv cons.$ in $V$, which is a contradiction with that $u$ is nontrivial.
	
	In case (II-3), by taking the same argument as in case (II-1) and letting
	$$t=1+\frac{1}{i},$$
	we finish the proof of Theorem \ref{thm1} (II).
\end{proof}

\begin{proof}[\rm\textbf{Proof of Theorem \ref{thm1} (\uppercase\expandafter{\romannumeral3})}]
	Let us divide the proof into three cases:
	\begin{enumerate}
		\item[(III-1).]{$(p,q)\in G_3 \cap \{p+q>1\}$;}\\
		\item[(III-2).]{$(p,q)\in G_3 \cap \{p+q=1\};$}\\
		\item[(III-3).]{$(p,q)\in G_3 \cap \{p+q<1\}.$}\\
	\end{enumerate}
	
	In case (III-1), since $(p,q)\in K_3$, we take
	\begin{align*}
		t=-\frac{p}{q-1}+\frac{1}{i}.
	\end{align*}
	and $b$ to be some large fixed constant.
	
	Letting $\Omega=V$ in lemma \ref{lem2}, substituting $\varphi=\varphi_i$ to (\ref{est-2}), and using the same procedure as before, we obtain
	\begin{align*}
		\sum\limits_{x \in V}\mu(x)  u(x)^{p-t} |\nabla u(x)|^q \varphi_i(x)^s
		\lesssim  i^{-\frac{2p+q+t(q-2)}{p+q-1}}
		\sum\limits^{2i}_{k=i-1}
		\mu(B_k) 2^{-k\frac{2p+q+t(q-2)}{p+q-1}}.
	\end{align*}
	Combining with (\ref{vol-3}), we obtain
	\begin{align}\label{3-1-10}
		&\sum\limits_{x \in V}\mu(x)  u(x)^{\frac{pq}{q-1}+\frac{1}{i}} |\nabla u|^q \varphi_i(x)^s
		\nonumber \\ &
		\lesssim  i^{-\frac{2p+q+t(q-2)}{p+q-1}}
		\sum\limits^{2i}_{k=i-1}
		2^{k\left(\frac{q}{q-1}-\frac{2p+q+t(q-2)}{p+q-1}\right)}k^{\frac{1}{q-1}}
		\nonumber \\ &
		\lesssim  i^{-\frac{2p+q+t(q-2)}{p+q-1}}
		\sum\limits^{2i}_{k=i-1}
		2^{-\frac{k(q-2)}{i(p+q-1)}}k^{\frac{1}{q-1}}
		\nonumber \\ &
		\lesssim  i^{\frac{1}{q-1}-\frac{2p+q+t(q-2)}{p+q-1}+1}
		\nonumber \\ &
		=i^{-\frac{q-2}{i(p+q-1)}},
	\end{align}
	where we have used that
	$$\frac{q}{q-1}-\frac{2p+q+t(q-2)}{p+q-1}=-\frac{q-2}{i(p+q-1)}.$$
	Then letting $i\to \infty$ in (\ref{3-1-10}), we obtain
	$$\sum\limits_{x \in V}\mu(x)  u(x)^{\frac{pq}{q-1}} |\nabla u(x)|^q <\infty.$$
	Repeating the same procedure as in proof of Theorem \ref{thm1} (I), we derive
	$$\sum\limits_{x \in V}\mu(x)  u(x)^{\frac{pq}{q-1}} |\nabla u(x)|^q=0,$$
	which contradicts with that $u$ is a nontrivial positive solution.
	
	In case (III-2), we take the same procedure as in case (II-2) except letting
	$0<\epsilon<-\frac{p}{2}$, thus $ (p^{\prime},q)=(p+\epsilon,q) \in$ (III-1).
	
	In case (III-3), we repeat the same argument as in case (III-1) except
 taking
	\begin{align*}
		t=-\frac{p}{q-1}-\frac{1}{i}.
	\end{align*}
	Hence, we complete proof of Theorem \ref{thm1} (III).
\end{proof}
\begin{proof}[\rm\textbf{Proof of Theorem \ref{thm1} (IV)}]
	Since here $p<0, q=1$, we choose
	$$t=l+\frac{1}{i},\qquad s=-\frac{l}{p}+2+\frac{1}{i},$$
	where $l>1$ is to be chosen later.
	
	Letting $\Omega=V$ in Lemma \ref{lem2}, substituting $\varphi=\varphi_i$  into (\ref{est-2}), and repeating the same procedure, we obtain
	\begin{align*}
		&\sum\limits_{x \in V}\mu(x)  u(x)^{p-t} |\nabla u(x)|^q \varphi_i(x)^s
		\nonumber \\&
		\lesssim  i^{-\frac{2p+q-t}{p}}
		\sum\limits^{2i}_{k=i-1} \mu(B_k) 2^{-k\frac{2p+q-t}{p}}
		\nonumber \\&
		\lesssim  i^{-\frac{2p+q-t}{p}}
		\sum\limits^{2i}_{k=i-1} 2^{k(\alpha-\frac{2p+q-t}{p})}.
	\end{align*}
	Letting $l$ be a fixed large enough constant such that for all $i$
	$$\alpha-\frac{2p+q-t}{p}<0,$$
	we obtain
	\begin{align}\label{4-1-3}
		\sum\limits_{x \in V}\mu(x)  u(x)^{p-t} |\nabla u(x)|^q \varphi_i(x)^s
		\lesssim i^{1-\frac{2p+q-t}{p}}.
	\end{align}
	Further, we require that $l$ satisfies
	$$1-\frac{2p+q-t}{p}<c<0,$$
by letting $i\to \infty$ in (\ref{4-1-3}), we obtain
	$$\sum\limits_{x \in V}\mu(x)  u(x)^{p-l} |\nabla u(x)|^q=0.$$
	which contradicts with that $u$ is nontrivial. Hence we complete the proof for Theorem \ref{thm1} (IV).
\end{proof}
\begin{proof}[\rm\textbf{Proof of Theorem \ref{thm1} (V)}]
	We divide the proof into two cases:
	\begin{enumerate}
		\item[(V-1).]{$(p,q)\in \{p+q=1,p\geq 0,q>0\};$}
		\item[(V-2).]{$(p,q)\in \{p+q=1,p>1, q<0\}$.}
	\end{enumerate}
	From (\ref{ieq}), we have
	$$\sum\limits_{y \in V}\frac{\mu_{xy}}{\mu(x)}u(y)-u(x)
	+u(x)^{p} |\nabla u(x)|^{q}\leq 0,$$
	that is
$$\sum\limits_{y \in V}\frac{\mu_{xy}}{\mu(x)}u(y)\leq u(x)(1-u(x)^{p-1} |\nabla u(x)|^{q}),$$
	which implies
	\begin{align}\label{5-ieq}
		u(x)^{p-1} |\nabla u(x)|^{q}\leq 1.
	\end{align}
	
	In case (V-1), since $p+q=1$, and $q>0 $, we obtain
	\begin{align}\label{5-1}
		|\nabla u(x)|\leq u(x).
	\end{align}
	Combining this with (\ref{ieq}), noting $p\geq0$, we derive
	\begin{align*}
		\Delta u(x)+(u(x)^{-p} |\nabla u(x)|^{p}) u(x)^{p} |\nabla u(x)|^{q}\leq \Delta u(x)+u(x)^{p} |\nabla u(x)|^{q}\leq 0,
	\end{align*}
	which is
	\begin{align}
		\Delta u(x)+\left|\nabla u(x)\right|\leq 0.
	\end{align}
	
Set $\Omega'_k:=\{ x\in V|0<u(x)<k,|\nabla u(x)|\neq 0 \}$.
		Since $u$ is a nontrivial positive solution, thus there exists some large enough $k$. Such that $\mu( \Omega'_k)>0 $. Now let $v=\frac{u}{k}$, it follows that $v$ satisfies
	$$ \Delta v(x)+\left|\nabla v(x)\right|\leq0.$$
Noting $0<v<1$ on $\Omega'_k$, from (\ref{5-1}), we have
	$$\left|\nabla v(x)\right|\leq v(x)\leq 1.$$
It follows that
	\begin{align}\label{5-2}
		\Delta v(x)+\left|\nabla v(x)\right|^{\lambda}\leq0, \quad\mbox{ on $\Omega'_k$},
	\end{align}
	where $\lambda>1$ is to be chosen later.
	
	From the definition of $\Omega'_k$, we know when $\mu_{xy}\neq 0$, $y\in V \setminus \Omega'_k$ $x\in \Omega'_k$,  we have
	\begin{align*}
		\left\{
		\begin{array}{lr}
			v(y)>v(x), \quad \mbox{when $v(y)\geq1$,}\\
			v(y)=v(x), \quad \mbox{when $v(y)<1$, and $|\nabla v(y)|=0$.}
		\end{array}
		\right.
	\end{align*}
In both cases, we have $v(y)-v(x)\geq 0$.

	Then, for any $x\in \Omega'_k$, we obtain
	\begin{align}\label{5-3}
		\Delta v(x)&=\sum\limits_{y\in V} \frac{\mu_{xy}}{\mu(x)}(v(y)-v(x))\nonumber\\&
		=\sum\limits_{y\in \Omega'_k} \frac{\mu_{xy}}{\mu(x)}(v(y)-v(x))+
		\sum\limits_{y\in V \setminus \Omega'_k} \frac{\mu_{xy}}{\mu(x)}(v(y)-v(x))\nonumber\\&
		\geq \sum\limits_{y\in \Omega'_k} \frac{\mu_{xy}}{\mu(x)}(v(y)-v(x)).
	\end{align}
	Similarly
	\begin{align}\label{5-4}
		|\nabla v(x)|&=\sqrt{\sum\limits_{y \in V}\frac{\mu_{xy}}{2\mu(x)}(\nabla_{xy}v)^2} \nonumber\\&
		=\sqrt{\sum\limits_{y \in\Omega'_k}\frac{\mu_{xy}}{2\mu(x)}(\nabla_{xy}v)^2+
			\sum\limits_{y \in(\Omega'_k)^c}\frac{\mu_{xy}}{2\mu(x)}(\nabla_{xy}v)^2}
		\nonumber\\&
		\geq \sqrt{\sum\limits_{y\in \Omega'_k}\frac{\mu_{xy}}{2\mu(x)}(\nabla_{xy}v)^2}
		:=|\nabla_{\Omega'_k} v(x)|.
	\end{align}
	It is should be know that $|\nabla_{\Omega'_k} u(x)|$ is not the norm of gradient of $ u $ in $\Omega'_k$, since $\mu(x)$ there is still the measure of $x$ in $ V $, instead of in $\Omega$. Noticing $\Omega'_k$ is a  subset of $V$, we have $\sum\limits_{y\in \Omega}\mu_{xy}\leq\sum\limits_{y\in V}\mu_{xy}=\mu(x)$.
	
	Subsitituting (\ref{5-3}) and (\ref{5-4}) into (\ref{5-2}), we obtain
	\begin{align}\label{5-5}
		\sum\limits_{y\in \Omega'_k} \frac{\mu_{xy}}{\mu(x)}(v(y)-v(x))+|\nabla_{\Omega'_k} v(x)|^{\lambda}\leq0, \quad \mbox{on $\Omega'_k$}.
	\end{align}
	
	Multiplying (\ref{5-5}) by $\mu(x)h_n^z$ in (\ref{hn}) and summing up over $x \in \Omega'_k$, we have
	\begin{align}\label{5-6}
		\sum\limits_{x \in \Omega'_k}\mu(x)|\nabla_{\Omega'_k} v(x)|^{\lambda}h_n(x)^z
		&\leq -\sum\limits_{x,y \in \Omega'_k} \mu_{xy} h_n(x)^z(\nabla_{xy}v) \nonumber \\&
		=\frac{1}{2}\sum\limits_{x,y \in \Omega'_k} \mu_{xy} (\nabla_{xy}h_n^z)(\nabla_{xy}v) \nonumber \\&
		=\frac{z}{2}\sum\limits_{x,y \in \Omega'_k} \mu_{xy} \eta^{z-1}(\nabla_{xy}h_n)(\nabla_{xy}v),
	\end{align}
	where $ \eta>0$ is between $h_n(x)$ and $h_n(y)$.
	
	Noting that $|\nabla_{\Omega'_k} v(x)|^2=\sum\limits_{y\in \Omega'_k}\frac{\mu_{xy}}{2\mu(x)}(\nabla_{xy}v)^2$, and
	$ \frac{1}{2p_0} \leq \frac{\mu_{xy}}{2\mu(x)} \leq \frac{1}{2}$, we derive
	\begin{align}\label{grd2}
		|\nabla_{xy}v| \leq \sqrt{2p_0}|\nabla_{\Omega'_k} v(x)|\qquad \mbox{for any $y\sim x$.}
	\end{align}
	
	Combining (\ref{grd2})  and $ \eta^{z-1} \leq max(h_n(x)^{z-1},h_n(y)^{z-1})\leq h_n(x)^{z-1}+h_n(y)^{z-1} $ with (\ref{5-6}),  and applying H\"{o}lder's inequality, we obtain
	\begin{align}
		&\sum\limits_{x \in \Omega'_k}\mu(x)|\nabla_{\Omega'_k} v(x)|^{\lambda}h_n(x)^z
		\nonumber \\&
		\leq \frac{z}{2}\sum\limits_{x,y \in \Omega'_k} \mu_{xy}  (h_n(x)^{z-1}+h_n(y)^{z-1})(\nabla_{xy}h_n)(\nabla_{xy}v) \nonumber \\&
		=z\sum\limits_{x,y \in \Omega'_k} \mu_{xy}  h_n(x)^{z-1}(\nabla_{xy}h_n)(\nabla_{xy}v)		\nonumber \\&
		\leq z\sum\limits_{x,y \in \Omega'_k} \mu_{xy}  h_n(x)^{z-1}|\nabla_{xy}h_n||\nabla_{xy}v| \nonumber \\&
		\leq\sqrt{2p_0}z\sum\limits_{x,y \in \Omega'_k} \mu_{xy}  h_n(x)^{z-1}|\nabla_{xy}h_n||\nabla_{\Omega'_k} v(x)| \nonumber \\&
		\leq \sqrt{2p_0}z\left( \sum\limits_{x,y \in \Omega'_k} \mu_{xy}|\nabla_{\Omega'_k} v(x)|^{\lambda}h_n(x)^{z}\right)^{\frac{1}{\lambda}}
		\left( \sum\limits_{x,y \in \Omega'_k} \mu_{xy}|\nabla_{xy}h_n|^{z}\right)^{\frac{\lambda-1}{\lambda}}\nonumber \\&
			\leq \sqrt{2p_0}z\left( \sum\limits_{x \in \Omega'_k} \mu(x)|\nabla_{\Omega'_k} v(x)|^{\lambda}h_n(x)^{z}\right)^{\frac{1}{\lambda}}
		\left( \sum\limits_{x,y \in \Omega'_k} \mu_{xy}|\nabla_{xy}h_n|^{z}\right)^{\frac{\lambda-1}{\lambda}},
	\end{align}
	where we take
	$$z=\dfrac{\lambda}{\lambda-1}.$$
	
	By the boundedness of $\sum\limits_{x \in \Omega'_k}\mu(x)|\nabla_{\Omega'_k} v(x)|^{\lambda}h_n^z$, and $h_n=1$ in $ B(o,n) $, and (\ref{vol-5}), we obtain
	\begin{align}\label{5-7}
		&\sum\limits_{x \in \Omega'_k\cap B(o,n) }\mu(x)|\nabla_{\Omega'_k}v(x)|^{\lambda}
		\nonumber \\&
		\leq (\sqrt{2p_0}z)^z \sum\limits_{x,y \in \Omega'_k} \mu_{xy}|\nabla_{xy}h_n|^{z}
		\nonumber \\&
		\leq (\frac{\sqrt{2p_0}z}{n})^z \mu(B(o,2n))
		\nonumber \\&
		\lesssim (\frac{\sqrt{2p_0}z}{n})^z e^{2\kappa n},
	\end{align}
	where we have used that $|\nabla_{xy}h_n|\leq \frac{1}{n}$, and $|\nabla_{xy}h_n|=0$ for  $x,y\in B(o,2n)^c$.
	Set
	\begin{align}\label{5-8}
		z=\theta n,
	\end{align}
	where $\theta$ is a fixed positive constant to be determined later. It is easy to see that $\lambda \to 1_+$ is equivalent to $n \to \infty$.
	
	Now let $\lambda \to 1_+$ in (\ref{5-7}), by (\ref{5-8}), we obtain
	\begin{align}\label{5-9}
		\sum\limits_{x \in \Omega'_k }\mu(x)|\nabla_{\Omega'_k}v(x)|&\leq \lim_{\lambda \to 1_+}\sum\limits_{x \in \Omega'_k\cap B(o,n) }\mu(x)|\nabla_{\Omega'_k}v(x)|^{\lambda}
		\nonumber \\&
		\lesssim  \lim_{n\to \infty} (\frac{\sqrt{2p_0}z}{n})^z e^{2\kappa n}\nonumber \\&
		\asymp \lim_{n\to \infty} e^{n(2\kappa+\theta\ln(\sqrt{2p_0}\theta))}.
	\end{align}
	If we have
	\begin{align}\label{5-10}
		2\kappa+\theta\ln(\sqrt{2p_0}\theta)<0,
	\end{align}
	which is equivalent to
	$$e^{2\kappa}<\frac{1}{(\sqrt{2p_0}\theta)^\theta}.$$
	
	Since $\frac{1}{(\sqrt{2p_0}\theta)^\theta}$ attains its maximum at $\theta=\frac{1}{\sqrt{2p_0}e}$. Hence
	if $\kappa$ satisfies
	\begin{align}
		0<\kappa<\kappa_0=\frac{1}{2\sqrt{2p_0}e},
	\end{align}
	there always exists $\theta>0$ such that (\ref{5-10}) holds.
	
	Under the above choice of $\kappa$ and $\theta$, from (\ref{5-9}), we obtain
	\begin{align}\label{5-11}
		\sum\limits_{x \in \Omega'_k }\mu(x)|\nabla_{\Omega'_k}v(x)|=0.
	\end{align}
	
	Hence, $x\in \Omega'_k$, $|\nabla_{\Omega'_k}v(x)|=0$. We calim that $|\nabla v(x)|=0$ for any $x\in \Omega'_k$. Then (\ref{5-11}) contradicts with the definition of $\Omega'_k$.
	
	Now assume there exists some $x_0\in \Omega'_k$ satisfying $ |\nabla_{\Omega'_k}v(x_0)|=0 $ but $|\nabla v(x_0)|\not=0$.
	
	We define $ U=\{y| y\sim x_0 \mbox{ and } u(y) \not= u(x_0)\}$, it is easy to see $U\subset (\Omega'_k)^c$ since
	$ |\nabla_{\Omega'_k}v(x_0)|=0 $. For any point $y\in U$ , we derrive $v(y)>v(x_0)$. Otherwise $v(y)<v(x_0)$ by the definition of $U$, noting  $|\nabla v(y)|\geq \sqrt{\frac{\mu_{x_0y}}{2\mu(y)}(v(y)-v(x_0))^2}>0$,
	we derrive that $y\in \Omega'_k$, which contradicts with $y\in U$.
	
	Now we consider the laplacian of $v(x_0)$. Noticing $u(y)= u(x_0)$ for any $y\in U^c$, we can obtain
	\begin{align*}
		&\Delta v(x_0)=\sum\limits_{y\in V} \frac{\mu_{x_0y}}{\mu(x_0)}(v(y)-v(x_0))\nonumber\\&
		=\sum\limits_{y\in U} \frac{\mu_{x_0y}}{\mu(x_0)}(v(y)-v(x_0))+
		\sum\limits_{y\in U^c} \frac{\mu_{x_0y}}{\mu(x_0)}(v(y)-v(x_0))\nonumber\\&
		=\sum\limits_{y\in U} \frac{\mu_{x_0y}}{\mu(x_0)}(v(y)-v(x_0))>0,
	\end{align*}
	which contradicts with (\ref{5-2}).
	
	In case (V-2), since $p+q=1$ $q<0$, from (\ref{5-ieq}), we obtain
	\begin{align}\label{5-2-1}
		u(x)\leq|\nabla u(x)|.
	\end{align}
	
	Define $\Omega_k=\{ x\in V|0<u(x)<k \}$.
	Since u is a nontrivial positive solution, there exists a $k_0>0$ such that $\mu(\Omega_k)>0 $ for any $k>k_0$. Now fix $k$, let $v=\frac{u}{k}$, which satisfies
	$$ \Delta v+v^p\left|\nabla v\right|^q\leq0, \quad\mbox{ on $V$}.$$
	It is easy to see that $0<v<1$ in $\Omega_k$, hence
	\begin{align*}
		\Delta v+v^{p+\epsilon}\left|\nabla v\right|^q\leq0, \quad\mbox{ on $\Omega_k$},
	\end{align*}
	where $\epsilon>0$ is to be determined later.
		
	Since $(p',q):=(p+\epsilon,q)\in K_2$, $p'+q=1+\epsilon$, we fix $0<t<1$, take $s=\frac{2p+q+t(q-2)+2\epsilon}{\epsilon}+1$ and set
	$$z=\frac{2p+q+t(q-2)+2\epsilon}{\epsilon}$$
	It is easy to see that $z\to\infty$ is equivalent to $\epsilon\to 0$.
	
	Noting $\frac{1}{p_0}\leq\frac{v(x)}{v(y)} \leq p_0$ and $v(y)-v(x)\geq 0$ when $x\in \Omega_k$ and $y\in \Omega_k^c$, letting $\Omega=\Omega_k$ in Lemma \ref{lem2}, and substituting $\varphi=h_n$ in (\ref{est-2}), we get
	\begin{align}\label{5-2-3}
		&\sum\limits_{x \in \Omega_k}\mu(x)  v(x)^{p-t+\epsilon} |\nabla v(x)|^q h_n(x)^s
		\nonumber \\&
		\lesssim C'_{p_0,t}  (2s)^{\frac{2p'+q+t(q-2)}{p'+q-1}} t^{-\frac{p'+t(q-1)}{p'+q-1}}
		\sum\limits_{x,y\in \Omega_k}
		\mu_{xy} |\nabla_{xy}h_n|^{\frac{2p'+q+t(q-2)}{p'+q-1}}
		\nonumber \\&
		=C'_{p_0,t}  (2s)^{\frac{2p+q+t(q-2)+2\epsilon}{\epsilon}} t^{-\frac{p+t(q-1)+\epsilon}{\epsilon}}
		\sum\limits_{x,y\in \Omega_k}
		\mu_{xy} |\nabla_{xy}h_n|^{\frac{2p+q+t(q-2)+2\epsilon}{\epsilon}}
		\nonumber \\&
		=2^{\frac{p+t(q-1)+\epsilon}{2p+q+t(q-2)+2\epsilon}}(C_{p_0,t})^{\frac{p'+q-t}{2p+q+t(q-2)+2\epsilon}}
		(2C_{p_0,t}(z+1))^{z} (\frac{1}{t})^{z}
		\sum\limits_{x,y\in \Omega_k}
		\mu_{xy} |\nabla_{xy}h_n|^{z},
		\end{align}
	where $C_{p_0,t}=\frac{(\sqrt{2p_0}(1+p^t_0))^{\frac{p+ t(q-1)+\epsilon}{p + q-t+\epsilon}+1}   (p_0^{t+1})^{\frac{p + t(q-1)+\epsilon}{p+ q-t+\epsilon}}}{4}$, $C'_{p_0,t} =( C_{p_0,t})^{\frac{p+q-t+\epsilon}{p+q-1+\epsilon}}$.
	
	It is easy to see that $C_{p_0,t}\to C_1$  and $2^{\frac{p+t(q-1)+\epsilon}{2p+q+t(q-2)+2\epsilon}}(C_{p_0,t})^{\frac{p'+q-t}{2p+q+t(q-2)+2\epsilon}}\to C_2$ as $\epsilon\to 0$, where $C_1=\frac{(\sqrt{2p_0}(1+p^t_0))^{\frac{p+ t(q-1)}{p + q-t}+1}   (p_0^{t+1})^{\frac{p + t(q-1)}{p+ q-t}}}{4}$ and $C_2$ are two positive constant depending on $p$, $q$ and $p_0$.
	
	Combining (\ref{5-2-3}) with $h_n=1$ in $ B(o,n) $, we obtain
	\begin{align}\label{5-2-4}
		&\sum\limits_{x \in \Omega'_k\cap B(o,n) }\mu(x) v(x)^{p-t+\epsilon} |\nabla v(x)|^q
		\nonumber \\&
		\lesssim 	(\frac{2C_1(z+1)}{t})^{z} \sum\limits_{x,y\in \Omega_k}\mu_{xy} |\nabla_{xy}h_n|^{n}
		\nonumber \\&
		\leq (\frac{2C_1(z+1)}{tn})^{z} \mu(B(o,2n))
		\nonumber \\&
		\lesssim (\frac{2C_1(z+1)}{tn})^{z} e^{2\kappa n}.
	\end{align}
	Now let us connect z and n by defining
	\begin{align}
		z=\theta n,
	\end{align}
	where $\theta$ is a fixed positive constant to be determined later.
	
	By let $n\to \infty$ in (\ref{5-2-4}), we obtain
	\begin{align}\label{5-2-5}
		\sum\limits_{x \in \Omega'_k}\mu(x) v(x)^{p-t}  |\nabla v(x)|^q &
		\lesssim  \lim_{n\to \infty} (\frac{2C_1(\theta n+1)}{tn})^{\theta Nn} e^{2\kappa n}\nonumber \\&
		\asymp \lim_{n\to \infty} e^{n(2\kappa+\theta\ln(\frac{2\theta C_1}{t}))}.
	\end{align}
	
	Similarly, if $\kappa$ satisfies
	\begin{align}
		0<\kappa<\kappa_{0,t}=\frac{t}{4C_1e},
	\end{align}
	
	there always exists $\theta>0$ such that $2\kappa+\theta\ln(\frac{2\theta C_1}{t})<0$.
	
	Noticing $p+q=1$, since $\frac{p+t(q-1)}{p+q-t}=1-q$, we have
	$$C_1=\frac{(\sqrt{2p_0}(1+p^t_0))^{\frac{p+ t(q-1)}{p + q-t}+1}   (p_0^{t+1})^{\frac{p + t(q-1)}{p+ q-t}}}{4}=\frac{(\sqrt{2p_0}(1+p^t_0))^{2-q}   (p_0^{t+1})^{1-q}}{4}.$$
	
	Since $\frac{t}{4C_1e}=\frac{t}{(e\sqrt{2p_0}(1+p^t_0))^{2-q}   (p_0^{t+1})^{1-q}}$, $t\in[0,1]$ attains its maximum at $t=1$. Then for any $0<\kappa<\kappa_0:=\frac{1}{(e\sqrt{2p_0}(1+p_0))^{2-q}   (p_0^{2})^{1-q}}$, there exists some $0<t_{\kappa}<1$
such that $\kappa<\kappa_{0,t}$ and $2\kappa+\theta\ln(\frac{2\theta C_1}{t})<0$.	
	
	Under the above choice of $\kappa$ and $\theta$ , from (\ref{5-2-5}), we obtain
	\begin{align}\label{5-2-6}
		\sum\limits_{x \in \Omega'_k}\mu(x) v(x)^{p-t_{\kappa}}  |\nabla v(x)|^q=0,
	\end{align}
	which contradicts with that u is a nontrivial positive solution by (\ref{5-2-1}).
	Hence we complete the proof for Theorem \ref{thm1} (V).
\end{proof}

\begin{proof}[\rm\textbf{Proof of Theorem \ref{thm1-1}}]
	Since $u$ is nontrivial, there exists some $x_0\in V$ such that $|\nabla u(x_0)|\neq 0$. Let $x_1$ be a neighboring vertex of $x_0$ such that
$$u(x_1)=min_{y \sim x_0}u(y).$$
	
	By (\ref{ieq}), and noting $u$ is nontrivial, we obtain $\Delta u(x_0)<0$.
Since $|\nabla u(x_0)|\neq 0$, we have
	$$u(x_0)> u(x_1), \quad|\nabla u(x_1)|\neq 0,\quad \text{and}\quad \Delta u(x_1)<0.$$
	Inductively, for $x_i$, we can find $x_{i+1}\in V$ such that $u(x_{i+1})=\min\limits_{y \sim x_i}u(y)$ and
	$$u(x_i)> u(x_{i+1}),\quad|\nabla u(x_{i+1})|\neq 0,\quad\text{and}\quad \Delta u(x_{i+1})<0. $$
	Combining $u(x_{i+1})< u(x_i)<\cdot\cdot\cdot< u(x_0)$ with $u(x)>0$, by Monotone Convergence theorem, we obtain there exists nonegative constant $u_0$ such that
	\begin{align}\label{6-lim}
		\lim\limits_{i\to \infty}u(x_i)=u_0.
	\end{align}
	Since
\begin{equation}\label{6-lap}
	0>\Delta u(x_i)=\sum\limits_{y \sim x_i} \frac{\mu_{x_iy}}{\mu(x_i)}u(y)-u(x_i)\geq u(x_{i+1})-u(x_i).
\end{equation}
	by Squeeze theorem, we know
	\begin{align}\label{6-lim-1}
		\lim_{i\to \infty}\Delta u(x_i)=0.
	\end{align}
	It follows from (\ref{ieq}) that
	\begin{align}\label{6-lim-2}
		\lim_{i\to \infty}u(x_i)^p|\nabla u(x_i)|^q=0.
	\end{align}
	Applying Jesnsen's inequality, we obtain
	\begin{align*}
		|\Delta u(x_i)|^2&=\left| \sum\limits_{y \sim x_i}\frac{\mu_{x_iy}}{\mu(x_i)}(u(y)-y(x_i))\right|^2\nonumber \\&
		\leq\left( \sum\limits_{y \sim x_i}\frac{\mu_{x_iy}}{\mu(x_i)}|u(y)-y(x_i)|\right)^2\nonumber \\&
		\leq \sum\limits_{y \sim x_i}\frac{\mu_{x_iy}}{\mu(x_i)}(u(y)-y(x_i))^2=2|\nabla u(x_i)|^2,
	\end{align*}
Consequently
	\begin{align}\label{6-ieq}
		|\Delta u(x_i)|\leq\sqrt{2}|\nabla u(x_i)|.
	\end{align}
	Combining (\ref{ieq}) with (\ref{6-ieq}), we derive
	\begin{align*}
		u(x_i)^p|\nabla u(x_i)|^q\leq -\Delta u(x_i)\leq|\Delta u(x_i)|\leq\sqrt{2}|\nabla u(x_i)|.
	\end{align*}
	Noticing $u(x_i)>0$ and $|\nabla u(x_i)|\not=0$, we obtain
	\begin{align}\label{6-ieq-2}
		1\leq	\sqrt{2}u(x_i)^{-p}|\nabla u(x_i)|^{1-q}.
	\end{align}

Let us finish the proof by dividing into the following cases:
	\begin{enumerate}
	\item[(1).]{$q=0$, $0<p\leq1$}
	\item[(2).]{$q=p=0$}
	\item[(3).]{$q=0, p<0$}
	\item[(4).]{$0<q<1, p<1-q$}
	\item[(5).]{$q<0$, $0<p<1-q$}
	\item[(6).]{$q<0, p=0$}
	\item[(7).]{$q<0, p<0$}
\end{enumerate}

In case (1), since $q=0$, $0<p\leq 1$, combining (\ref{6-lim}) and (\ref{6-lim-2}),
we obtain
$$\lim_{i\to \infty}u(x_i)=0.$$
By (\ref{ieq}), we get
\begin{align*}
	0\geq\sum\limits_{y \sim x_i} \frac{\mu_{x_iy}}{\mu(x_i)}(u(y)-u(x_i))+u(x_i)^{p}.
\end{align*}
For $0<p< 1$, we have
\begin{align*}
	0\geq -u(x_i)+u(x_i)^{p},
\end{align*}
which implies $u(x_i)\geq1$, and contradicts with $\lim\limits_{i\to \infty}u(x_i)=0$.

For $p= 1$, we have
\begin{align*}
	0\geq\sum\limits_{y \sim x_i} \frac{\mu_{x_iy}}{\mu(x_i)}u(y)-u(x_i)+u(x_i)=\sum\limits_{y \sim x_i} \frac{\mu_{x_iy}}{\mu(x_i)}u(y),
\end{align*}
which contradicts with $u>0$.

In case (2), by (\ref{ieq}), we have
\begin{align*}
	\Delta u(x_i)+1\leq 0,
\end{align*}
which contradicts with (\ref{6-lim-1})

In case (3), since $q=0$, $p<0$, we derive a contradiction from (\ref{6-lim})-(\ref{6-lim-2}).

In case (4), we divide the proof into two cases:
\begin{enumerate}
	\item[(4-1).]{there exists some $k_0$, such that when $i>k_0$, $|\nabla u(x_i)|\leq \lambda(u(x_i)-u(x_{i+1}))$, where $\lambda=u(x_0)^{\frac{1-p-q}{q}}$}.
	\item[(4-2).]{there exists a series $\{i_k\}$, such that $i_k\to\infty$ as $k\to \infty$ and $|\nabla u(x_{i_k})|> \lambda(u(x_{i_k})-u(x_{i_k+1}))$.}
\end{enumerate}

In case (4-1), since $1-q>0$, by substituting $|\nabla u(x_i)|\leq \lambda(u(x_i)-u(x_{i+1}))$ into (\ref{6-ieq-2}), we obtain
$$1\leq \sqrt{2}\lambda^{1-q}u(x_i)^{-p}(u(x_i)-u(x_{i+1}))^{1-q},$$
which implies by (\ref{6-lim}) that $p>0$ and $\lim\limits_{i\to \infty}u(x_i)=u_0=0$.

Noticing $u(x_i)-u(x_{i+1})<u(x_i)$ and $0<p<1-q$, we have
$$1\leq \sqrt{2}\lambda^{1-q}u(x_i)^{1-q-p}.$$
Since $p<1-q$, by letting $i\to \infty$, we derive a contradiction.

In case (4-2), and by (\ref{ieq}), we have
$$\Delta u(x_{i_k})+u(x_{i_k})^p|\nabla u(x_{i_k})|^q\leq 0,$$
which is equivalent to
$$|\nabla u(x_{i_k})|\leq \left[ |\Delta u(x_{i_k})|u(x_{i_k})^{-p}\right]^{\frac{1}{q}}. $$

Noting $0<q<1$, by (\ref{6-lap}), (\ref{6-lim-1}) and $|\nabla u(x_{i_k})|> \lambda (u(x_{i_k})-u(x_{i_k+1}))$, we obtain
$$\lambda (u(x_{i_k})-u(x_{i_k+1}))\leq |\nabla u(x_{i_k})|\leq (u(x_{i_k})-u(x_{i_k+1}))^{\frac{1}{q}}u(x_{i_k})^{-\frac{p}{q}},$$
which implies
$$\lambda\leq (u(x_{i_k})-u(x_{i_k+1}))^{\frac{1}{q}-1}u(x_{i_k})^{-\frac{p}{q}}\leq u(x_{i_k})^{\frac{1-p-q}{q}} .$$
Noting $\lambda=u(x_0)^{\frac{1-p-q}{q}}$, and $u(x_i)<u(x_0)$ for $i\geq1$, we derive a contradiction.

	For case (5), (6), and (7), since
\begin{align*}
	0\geq\Delta u(x_i)=\sum_{y\sim x_i\atop u(y)>u(x_i)}\frac{\mu_{x_iy}}{\mu(x_i)}(u(y)-u(x_i))+\sum_{y\sim x_i\atop u(y)\leq u(x_i)}\frac{\mu_{x_iy}}{\mu(x_i)}(u(y)-u(x_i)),
\end{align*}
It follows that
\begin{align*}
	\sum_{y\sim x_i\atop u(y)>u(x_i)}\frac{\mu_{x_iy}}{\mu(x_i)}(u(y)-u(x_i))\leq\sum_{y\sim x_i\atop u(y)\leq u(x_i)}\frac{\mu_{x_iy}}{\mu(x_i)}(u(x_i)-u(y)).
\end{align*}
Hence, we know for any $y_0\sim x_i$, and $u(y_0)>u(x_i)$, and using $(p_0)$ condition, we have
\begin{align}\label{6-5}
	\frac{1}{p_0}(u(y_0)-u(x_i))&\leq\frac{\mu_{x_iy_0}}{\mu(x_i)}((u(y_0)-u(x_i))\nonumber\\
	&\leq	\sum_{y\sim x_i\atop u(y>u(x_i))}\frac{\mu_{x_iy}}{\mu(x_i)}(u(y)-u(x_i))\nonumber\\
	&\leq \sum\limits_{y \sim x_i \atop u(y)\leq  u(x_i)}\frac{\mu_{x_iy}}{\mu(x_i)}(u(x_i)-u(y)) \leq u(x_i)-u(x_{i+1}).
\end{align}
Here we have also used that $u(x_{i+1})=\min\limits_{y\sim x_i}u(y)<u(x_i)$.

Combining (\ref{6-5}) with the definition of $|\nabla u(x_i)|$, we have
\begin{align*}
	|\nabla u(x_i)|^2&=\sum\limits_{y \sim x_i \atop u(y)> u(x_i)}\frac{\mu_{x_iy}}{2\mu(x_i)}(u(y)-u(x_i))^2+\sum\limits_{y \sim x_i \atop u(y)\leq u(x_i)}\frac{\mu_{x_iy}}{2\mu(x_i)}(u(y)-u(x_i))^2
	\\ &
	\leq \frac{p_0^2}{2}(u(x_{i+1})-u(x_i))^2+\frac{1}{2}(u(x_{i+1})-u(x_i))^2
	\\ &
	= \frac{1+p_0^2}{2}(u(x_{i+1})-u(x_i))^2,
\end{align*}
which is
\begin{align}\label{6-6}
	|\nabla u(x_i)|\leq \sqrt{\frac{1+p_0^2}{2}}(u(x_i)-u(x_{i+1})),
\end{align}
It follows that
\begin{align}\label{6-lim-3}
	\lim_{i\to \infty}|\nabla u(x_i)|=0.
\end{align}.

	In case (5), noticing $p>0$, $q< 0$ and combining (\ref{6-lim}) (\ref{6-lim-2}) and (\ref{6-lim-3}), we derive
$$\lim_{i\to \infty}u(x_i)=0.$$
	By (\ref{ieq}), we have
	\begin{align*}
		0&\geq\sum\limits_{y \sim x_i} \frac{\mu_{x_iy}}{\mu(x_i)}(u(y)-u(x_i))+u(x_i)^{p}|\nabla u(x_i)|^q\\&
		=\sum\limits_{y \sim x_i} \frac{\mu_{x_iy}}{\mu(x_i)}u(y)-u(x_i)(1-u(x_i)^{p-1}|\nabla u(x_i)|^q),
	\end{align*}
	which implies
	\begin{align}\label{6-0}
		1-u(x_i)^{p-1}|\nabla u(x_i)|^q\geq 0.
	\end{align}
	Combining (\ref{6-0}) with (\ref{6-6}), we obtain
	\begin{align}\label{6-7}
		1\geq u(x_i)^{p-1}|\nabla u(x_i)|^q\geq (\frac{p_0^2+1}{2})^{\frac{q}{2}}u(x_i)^{p-1}(u(x_i)-u(x_{i+1}))^q.
	\end{align}
 Since $q<0$, we have
 \begin{align}\label{6-8}
 	1\geq (\frac{p_0^2+1}{2})^{\frac{q}{2}}u(x_i)^{p-1}(u(x_i)-u(x_{i+1}))^q\geq(\frac{p_0^2+1}{2})^{\frac{q}{2}} u(x_i)^{p+q-1}.
 \end{align}
	 Using $p<1-q$, and by letting $i\to \infty$, we obtain a contradiction from (\ref{6-8}) with $\lim\limits_{i\to \infty}u(x_i)=0$.

	In case (6), noticing $p=0$, $q<0$, we obtain a contradiction by (\ref{6-lim-2}) and (\ref{6-lim-3}).
	
	In case (7), since $p<0$, $q< 0$, substituting (\ref{6-lim}) and (\ref{6-lim-3}) into
	(\ref{6-lim-2}), we derive a contradiction.
	Hence, we complete the proof of Theorem \ref{thm1-1}.
\end{proof}

\section{Proof of Theorem \ref{thm2}}

Before presenting the proof of Theorem \ref{thm2}, for our convenience, let us introduce some notations. For fixed integer $n\geq0$, let us denote by $D_n: = \{x \in T_N : d(o, x) = n\}$ the collection of all the vertices with distance $n$ from $o$, and denote by $E_n$ the collection of all the edges from vertices in $D_n$ to vertices in $D_{n+1}$.

\begin{proof}[\rm{Proof of Theorem \ref{thm2} (I)}]
	When $(p, q)\in G_1$, let us take $\mu$ and $u$ as follows
	\begin{align}\label{mu1}
		&\mu_{xy}=\mu_n=\frac{(n+n_0)^{\frac{p+1}{p+q-1}}(\ln{(n+n_0)})^{\frac{1}{p+q-1}+\epsilon}}{(N-1)^n},\quad \mbox{for any $(x,y)\in E_n$, $n\geq0$}.
	\end{align}
	\begin{align}\label{u1}
		&u(x)=u_n=\frac{\delta}{(n+n_0)^{\frac{2-q}{p+q-1}}(\ln{(n+n_0)})^{\frac{1}{p+q-1}}},
		\quad \mbox{for any $x\in D_n$, $n\geq0$}.
	\end{align}
	where $n_0 \geq 2$  and $\delta > 0$ are to be chosen later.
	
First, under the above choice of $\mu$,
for $n \geq 2$, we obtain
	$$\mu(B(o,n))=\sum\limits^n_{k=0}\mu(D_k)\asymp \sum\limits^n_{k=0}(N-1)^k \mu_k\asymp  n^{\frac{2p+q}{p+q-1}}(\ln{n})^{\frac{1}{p+q-1}+\epsilon},$$
which implies that (\ref{e-vol-1}) holds.

Next we verify that (\ref{ieq}) holds for the above choice of $\mu$ and $u$, namely the following two inequalities hold:
	\begin{align}\label{en0}
		u_1-u_0+u_0^p\left[\frac{(u_0-u_1)^2}{2}\right]^{\frac{q}{2}} \leq 0,
	\end{align}
and
	\begin{align}\label{en}
		&\frac{(N-1)\mu_n u_{n+1}+\mu_{n-1} u_{n-1}}{(N-1)\mu_n+\mu_{n-1}}-u_n\nonumber\\
&+u_n^p\left[\frac{(N-1)\mu_n(u_{n+1}-u_n)^2+\mu_{n-1}(u_{n-1}-u_n)^2}{2(N-1)\mu_n+2\mu_{n-1}}\right]^{\frac{q}{2}} \leq 0.
	\end{align}
For brevity, we denote
$$\lambda=\frac{p+1}{p+q-1},\quad \beta=\frac{1}{p+q-1}, \quad\sigma=\frac{2-q}{p+q-1}.$$
Now let us deal with cases of $n = 0$ and $n\geq1$.
	
	\textbf{Case of $n = 0$}.  Combining with (\ref{u1}) and (\ref{mu1}), then (\ref{en0}) is equivalent to
	\begin{align*}
		&\frac{\delta}{(n_0+1)^{\sigma}(\ln{(n_0+1)})^{\beta}}-\frac{\delta}{n_0^{\sigma}(\ln{n_0})^{\beta}}+\left( \frac{\delta}{n_0^{\sigma}(\ln{n_0})^{\beta}}\right) ^p \nonumber \\&
		\times \left[ \frac{1}{2}\left( \frac{\delta}{n_0^{\sigma}(\ln{n_0})^{\beta}}-\frac{\delta}{(n_0+1)^{\sigma}(\ln{(n_0+1)})^{\beta}}\right) ^2\right] ^{\frac{q}{2}} \leq 0,
	\end{align*}
	the above is satisfied if we choose $\delta \leq \delta_0$ with
	\begin{align*}
		\delta_0&= 2^{\frac{q}{2(p+q-1)}} \left[\frac{1}{n_0^{\sigma}(\ln{n_0})^{\beta}}-\frac{1}{(n_0+1)^{\sigma}(\ln{(n_0+1)})^{\beta}} \right]^{\frac{1}{p+q-1}} \\ &
		\times\left( n_0^{\sigma}(\ln{n_0})^{\beta}\right)^{\frac{p}{p+q-1}} \left[ \frac{1}{n_0^{\sigma}(\ln{n_0})^{\beta}}-\frac{1}{(n_0+1)^{\sigma}(\ln{(n_0+1)})^{\beta}}\right]^{-\frac{q}{p+q-1}}.
	\end{align*}
	
	\textbf{Case of $n\geq1$}. Combining with (\ref{u1}) and (\ref{mu1}), then (\ref{en}) is equivalent to
	\begin{align*}
		&\frac{\delta\frac{(n+n_0)^{\lambda}(\ln{(n+n_0)})^{\beta+\epsilon}}{(n+n_0+1)^{\sigma}(\ln{(n+n_0+1)})^{\beta}}+\delta\frac{(n+n_0-1)^{\lambda}(\ln{(n+n_0-1)})^{\beta+\epsilon}}{(n+n_0-1)^{\sigma}(\ln{(n+n_0-1)})^{\beta}}}{(n+n_0)^{\lambda}(\ln{(n+n_0)})^{\beta+\epsilon}+(n+n_0-1)^{\lambda}(\ln{(n+n_0-1)})^{\beta+\epsilon}}	\nonumber \\ &
		-\frac{\delta}{(n+n_0)^{\sigma}(\ln{(n+n_0)})^{\beta}}+\left(\frac{\delta}{(n+n_0)^{\sigma}(\ln{(n+n_0)})^{\beta}} \right)^p \nonumber \\ &
		\times\delta^q\left[ \frac{(n+n_0)^{\lambda}(\ln{(n+n_0)})^{\beta+\epsilon}\left(\frac{1}{(n+n_0)^{\sigma}(\ln{(n+n_0)})^{\beta}}-\frac{1}{(n+n_0+1)^{\sigma}(\ln{(n+n_0+1)})^{\beta}} \right)^2 }{2(n+n_0)^{\lambda}(\ln{(n+n_0)})^{\beta+\epsilon}+2(n+n_0-1)^{\lambda}(\ln{(n+n_0-1)})^{\beta+\epsilon}}\right.\nonumber \\& \left.
		+\frac{(n+n_0-1)^{\lambda}(\ln{(n+n_0-1)})^{\beta+\epsilon}\left(\frac{1}{(n+n_0-1)^{\sigma}(\ln{(n+n_0-1)})^{\beta}}-\frac{1}{(n+n_0)^{\sigma}(\ln{(n+n_0)})^{\beta}} \right)^2 }{2(n+n_0)^{\lambda}(\ln{(n+n_0)})^{\beta+\epsilon}+2(n+n_0-1)^{\lambda}(\ln{(n+n_0-1)})^{\beta+\epsilon}}
		\right]^\frac{q}{2}\leq0,
	\end{align*}
	The above is equivalent to the following estimate holds  for all $n\geq1$
\begin{align}\label{3-1-1}
\delta^{p+q-1}\leq \Lambda_1(n+n_0),
\end{align}
where
	\begin{align}\label{3-1-2}
		\Lambda_1(n):=&(n^{\sigma}(\ln{n})^{\beta})^{p-1}\left[ 1-\frac{n^{\sigma}(\ln{n})^{\beta}\left( \frac{n^{\lambda}(\ln{n})^{\beta+\epsilon}}{(n+1)^{\sigma}(\ln{(n+1)})^{\beta}}+\frac{(n-1)^{\lambda}(\ln{(n-1)})^{\beta+\epsilon}}{(n-1)^{\sigma}(\ln{(n-1)})^{\beta}}\right) }{n^{\lambda}(\ln{n})^{\beta+\epsilon}+(n-1)^{\lambda}(\ln{(n-1)})^{\beta+\epsilon}}\right] \nonumber \\
&
		\times\left[ \frac{n^{\lambda}(\ln{n})^{\beta+\epsilon}\left(\frac{1}{n^{\sigma}(\ln{n})^{\beta}}-\frac{1}{(n+1)^{\sigma}(\ln{(n+1)})^{\beta}} \right)^2 }{2n^{\lambda}(\ln{n})^{\beta+\epsilon}+2(n-1)^{\lambda}(\ln{(n-1)})^{\beta+\epsilon}}\right.\nonumber \\& \left.
		+\frac{(n-1)^{\lambda}(\ln{(n-1)})^{\beta+\epsilon}\left(\frac{1}{(n-1)^{\sigma}(\ln{(n-1)})^{\beta}}-\frac{1}{n^{\sigma}(\ln{n})^{\beta}} \right)^2 }{2n^{\lambda}(\ln{n})^{\beta+\epsilon}+2(n-1)^{\lambda}(\ln{(n-1)})^{\beta+\epsilon}}
		\right]^{-\frac{q}{2}}.
	\end{align}
Hence, we have
	\begin{align}\label{3-1-lim}
		\lim_{n\to \infty} \Lambda_1(n)=(\frac{2}{\sigma})^{\frac{q}{2}-1} \epsilon.
	\end{align}
The detailed calculation of (\ref{3-1-lim}) is as follows: by using the facts that
$$\frac{\ln{(n-1)}}{\ln{n}}=1-\frac{1}{n\ln{n}}-\frac{1}{2n^2\ln{n}}+o(\frac{1}{2n^2\ln{n}}),$$
and
	$$(1-\frac{1}{n})^{\alpha}=1-\frac{\alpha}{n}-\frac{\alpha(\alpha-1)}{2n^2}+O(\frac{1}{n^3}).$$
and let us first deal with term
	\begin{align}\label{A1}
		A_1(n):=&\left( \frac{n^{\lambda}(\ln{n})^{\beta+\epsilon}\left(\frac{1}{n^{\sigma}(\ln{n})^{\beta}}-\frac{1}{(n+1)^{\sigma}(\ln{(n+1)})^{\beta}} \right)^2 }{2n^{\lambda}(\ln{n})^{\beta+\epsilon}+2(n-1)^{\lambda}(\ln{(n-1)})^{\beta+\epsilon}}\right.\nonumber \\& \left.
		+\frac{(n-1)^{\lambda}(\ln{(n-1)})^{\beta+\epsilon}\left(\frac{1}{(n-1)^{\sigma}(\ln{(n-1)})^{\beta}}-\frac{1}{n^{\sigma}(\ln{n})^{\beta}} \right)^2 }{2n^{\lambda}(\ln{n})^{\beta+\epsilon}+2(n-1)^{\lambda}(\ln{(n-1)})^{\beta+\epsilon}}
		\right)^{-\frac{q}{2}},
	\end{align}
	which is equal to
	\begin{align}\label{3-1-3}
		2^{\frac{q}{2}}\left( \frac{1+(1-\frac{1}{n})^{\lambda}(\frac{\ln{(n-1)}}{\ln{n}})^{\beta+\epsilon}}{\left(\frac{1}{n^{\sigma}(\ln{n})^{\beta}}-\frac{1}{(n+1)^{\sigma}(\ln{(n+1)})^{\beta}} \right)^2+(1-\frac{1}{n})^{\lambda}(\frac{\ln{(n-1)}}{\ln{n}})^{\beta+\epsilon}\left(\frac{1}{(n-1)^{\sigma}(\ln{(n-1)})^{\beta}}-\frac{1}{n^{\sigma}(\ln{n})^{\beta}} \right)^2}\right) ^{\frac{q}{2}}.
	\end{align}
	Notice that
	\begin{align}\label{3-1-4}
		\frac{1}{n^{\sigma}(\ln{n})^{\beta}}-\frac{1}{(n+1)^{\sigma}(\ln{(n+1)})^{\beta}} &=\frac{1}{n^{\sigma}(\ln{n})^{\beta}}\left( 1-(1-\frac{1}{n+1})^{\sigma}(\frac{\ln{n}}{\ln{(n+1)}})^{\beta}\right) \nonumber\\&
		=\frac{1}{n^{\sigma}(\ln{n})^{\beta}}\left(\frac{\sigma}{n+1} +\frac{\beta}{(n+1)\ln{(n+1)}}+o(\frac{1}{n\ln{n}})\right)\nonumber\\&
		=\frac{1}{n^{\sigma}(\ln{n})^{\beta}}\left(\frac{\sigma}{n} +o(\frac{1}{n})\right).
	\end{align}
	Similarly
	\begin{align}\label{3-1-5}
		\frac{1}{(n-1)^{\sigma}(\ln{(n-1)})^{\beta}}-\frac{1}{n^{\sigma}(\ln{n})^{\beta}}
		=\frac{1}{n^{\sigma}(\ln{n})^{\beta}}\left(\frac{\sigma}{n} +o(\frac{1}{n})\right).
	\end{align}
	Combining (\ref{3-1-4}) and (\ref{3-1-5}) with (\ref{3-1-3}), we obtain
	\begin{align}\label{3-1-6}
		A_1(n)=(\frac{2}{\sigma})^{\frac{q}{2}}(n^{\sigma+1}(\ln{n})^{\beta})^q\left( \frac{2+o(1)}{2+o(1)}\right)^{\frac{q}{2}}
		\asymp (\frac{2}{\sigma})^{\frac{q}{2}}(n^{\sigma+1}(\ln{n})^{\beta})^q.
	\end{align}
	Substituting (\ref{3-1-6}) into (\ref{3-1-2}), we have
	\begin{align*}
		\lim_{n\to\infty}\Lambda_1(n) =&\lim_{n\to\infty} (\frac{2}{\sigma})^{\frac{q}{2}}(n^{\sigma}(\ln{n})^{\beta})^{p+q-1}n^q  \left( 1-\frac{n^{\sigma}(\ln{n})^{\beta}\left( \frac{n^{\lambda}(\ln{n})^{\beta+\epsilon}}{(n+1)^{\sigma}(\ln{(n+1)})^{\beta}}+\frac{(n-1)^{\lambda}(\ln{(n-1)})^{\beta+\epsilon}}{(n-1)^{\sigma}(\ln{(n-1)})^{\beta}}\right) }{n^{\lambda}(\ln{n})^{\beta+\epsilon}+(n-1)^{\lambda}(\ln{(n-1)})^{\beta+\epsilon}}\right)
		\nonumber\\
		= &\lim_{n\to\infty}(\frac{2}{\sigma})^{\frac{q}{2}}(n^{\sigma}(\ln{n})^{\beta})^{p+q-1}n^q  \left( 1-\frac{(1-\frac{1}{n+1})^{\sigma}(\frac{\ln{n}}{\ln{(n+1)}})^{\beta}
+(1-\frac{1}{n})^{\lambda-\sigma}(\frac{\ln{(n-1)}}{\ln{n}})^{\epsilon}}{1+(1-\frac{1}{n})^{\lambda}(\frac{\ln{(n-1)}}{\ln{n}})^{\beta+\epsilon}}\right)
\nonumber\\
=&\lim_{n\to\infty}(\frac{2}{\sigma})^{\frac{q}{2}}n^{2}\ln{n}
		\left( 1-\frac{(1-\frac{1}{n+1})^{\sigma}(\frac{\ln{n}}{\ln{(n+1)}})^{\beta}+(1-\frac{1}{n})(\frac{\ln{(n-1)}}
{\ln{n}})^{\epsilon}}{1+(1-\frac{1}{n})^{\lambda}(\frac{\ln{(n-1)}}{\ln{n}})^{\beta+\epsilon}}\right).
	\end{align*}
where we have used that $\lambda=\frac{p+1}{p+q-1}$, $\beta=\frac{1}{p+q-1}$, $\sigma=\frac{2-q}{p+q-1}=\lambda-1$.
	Applying Taylor expansion technique, we obtain
	\begin{align*}
		\lim_{n\to\infty}\Lambda_1(n)
		=(\frac{2}{\sigma})^{\frac{q}{2}-1} \epsilon,
	\end{align*}
	which yields (\ref{3-1-lim}).
	This implies that there exists some large enough $n_0$ such that for all $n\geq 0$, the RHS of (\ref{3-1-1}) is bounded from above.

	Finally, take $n_0$ as above and $\delta=\min\{\delta_0, \delta_1\}$, where we choose $\delta_1=\left[\frac{(\frac{2}{\sigma})^{\frac{q}{2}-1} \epsilon}{2}\right]^{\frac{1}{p+q-1}}$.
It follows that when $(p, q)\in G_1$, $u$ is a solution to (\ref{ieq}). Hence we complete the proof for Theorem \ref{thm2} (I).
\end{proof}
\vskip1ex

\begin{proof}[\rm{Proof of Theorem \ref{thm2} (II)}]
	When $(p, q)\in G_2$, take $\mu$ and $u$ as follows
	\begin{align}\label{mu2}
		&\mu_{xy}=\mu_n=\frac{(n+n_0)(\ln{(n+n_0)})^{1+\epsilon}}{(N-1)^n},\quad \mbox{for any $(x,y)\in E_n$, $n\geq0$},
	\end{align}
	\begin{align}\label{u2}
		&u(x)=u_n=\frac{1}{(\ln{(n+n_0)})^{\frac{\epsilon}{2}}}+1,
		\quad \mbox{for any $x\in D_n$, $n\geq0$}.
	\end{align}
	For $n \geq 2$, it is easy to verify that
	$$\mu(B(o,n))\asymp  n^{2}(\ln{n})^{1+\epsilon},$$
which yields (\ref{e-vol-2}) holds.

Now we need to verify that (\ref{en0}) and (\ref{en}) holds under the above choice of $\mu$ and $u$.
	
	\textbf{Case of $n = 0$}. Substituting (\ref{u2}) and (\ref{mu2}) into (\ref{en0}), we obtain
	\begin{align}\label{3-2-1}
		&\frac{1}{(\ln{(n_0+1)})^{\frac{\epsilon}{2}}}-\frac{1}{(\ln{n_0})^{\frac{\epsilon}{2}}}+\left( \frac{1}{(\ln{n_0})^{\frac{\epsilon}{2}}}+1\right) ^p\nonumber \\&
		\times \left[\frac{1}{2} \left( \frac{1}{(\ln{n_0})^{\frac{\epsilon}{2}}}-\frac{1}{(\ln{(n_0+1)})^{\frac{\epsilon}{2}}}\right)^2 \right] ^{\frac{q}{2}} \leq 0.
	\end{align}
	It is easy to verify that the above holds for $q\geq 2$ when $n_0$ is large enough.
	
	\textbf{Case of $n\geq1$}. By substituting (\ref{u2}) and (\ref{mu2}), we know (\ref{en}) is equivalent to
	\begin{align*}
		&\frac{\frac{(n+n_0)(\ln{(n+n_0)})^{1+\epsilon}}{ (\ln{(n+n_0+1)})^{\frac{\epsilon}{2}}}+\frac{(n+n_0-1)(\ln{(n+n_0-1)})^{1+\epsilon}}{(\ln{(n+n_0-1)})^{\frac{\epsilon}{2}}}}{(n+n_0)(\ln{(n+n_0)})^{1+\epsilon}+(n+n_0-1)(\ln{(n+n_0-1)})^{1+\epsilon}}	\nonumber \\ &
		-\frac{1}{(\ln{(n+n_0)})^{\frac{\epsilon}{2}}}+\left(\frac{1}{(\ln{(n+n_0)})^{\frac{\epsilon}{2}}}+1 \right)^p \nonumber \\ &
		\times \left( \frac{(n+n_0)(\ln{(n+n_0)})^{1+\epsilon}\left(\frac{1}{ (\ln{(n+n_0)})^{\frac{\epsilon}{2}}}-\frac{1}{(\ln{(n+n_0+1)})^{\frac{\epsilon}{2}}} \right)^2}{2(n+n_0)(\ln{(n+n_0)})^{1+\epsilon}+2(n+n_0-1)(\ln{(n+n_0-1)})^{1+\epsilon}}\right.\nonumber \\& \left.
		+\frac{(n+n_0-1)(\ln{(n+n_0-1)})^{1+\epsilon}\left(\frac{1}{(\ln{(n+n_0-1)})^{\frac{\epsilon}{2}}}-\frac{1}{(\ln{(n+n_0)})^{\frac{\epsilon}{2}}} \right)^2}{2(n+n_0)(\ln{(n+n_0)})^{1+\epsilon}+2(n+n_0-1)(\ln{(n+n_0-1)})^{1+\epsilon}}
		\right)^\frac{q}{2} \leq0,
	\end{align*}
	namely
\begin{equation}\label{3-2-2}
1\leq \Lambda_2(n+n_0),
\end{equation}
where
	\begin{align}\label{3-2-3}
		\Lambda_2(n):=&\left(1-\frac{(\ln{n})^{\frac{\epsilon}{2}}\left( \frac{n(\ln{n})^{1+\epsilon}}{(\ln{(n+1)})^{\frac{\epsilon}{2}}}+
			\frac{(n-1)(\ln{(n-1)})^{1+\epsilon}}{(\ln{(n-1)})^{\frac{\epsilon}{2}}}\right)}{n(\ln{n})^{1+\epsilon}+(n-1)(\ln{(n-1)})^{1+\epsilon}}\right)
		\left( \frac{1}{(\ln{n})^{\frac{\epsilon}{2}}}+1\right)^{-p} ((\ln{n})^{\frac{\epsilon}{2}})^{-1}\nonumber \\ &
		\times\left(
		\frac{2n(\ln{n})^{1+\epsilon}+2(n-1)(\ln{(n-1)})^{1+\epsilon}}{n(\ln{n})^{1+\epsilon}\left(\frac{1}{(\ln{n})^{\frac{\epsilon}{2}}}-\frac{1}{(\ln{(n+1)})^{\frac{\epsilon}{2}}} \right)^2+(n-1)(\ln{(n-1)})^{1+\epsilon}\left(\frac{1}{(\ln{(n-1)})^{\frac{\epsilon}{2}}}-\frac{1}{(\ln{n})^{\frac{\epsilon}{2}}} \right)^2}
		\right)^{\frac{q}{2}}.
	\end{align}
Applying the same argument as in proof of Theorem \ref{thm2} (I), we have
	\begin{align}\label{3-2-lim}
		\lim_{n\to \infty} \Lambda_2(n)=\infty.
	\end{align}
	This implies that there exists some large $n_0$ such that both of (\ref{3-2-1}) and (\ref{3-2-2}) hold, hence we complete the proof for Theorem \ref{thm2} (II).
\end{proof}
\vskip1ex

\begin{proof}[\rm{Proof of Theorem \ref{thm2} (III)}]
	When $(p, q)\in G_3$, take $\mu$ and $u$ as follows
	\begin{align}\label{mu3}
		&\mu_{xy}=\mu_n=\frac{(n+n_0)^{\frac{1}{q-1}}(\ln{(n+n_0)})^{\frac{1}{q-1}+\epsilon}}{(N-1)^n},\quad \mbox{for any $(x,y)\in E_n$, $n\geq0$}.
	\end{align}
	\begin{align}\label{u3}
		&u(x)=u_n=\frac{\delta}{(n+n_0)^{\frac{2-q}{q-1}}(\ln{(n+n_0)})^{\frac{1}{q-1}}}+1,
		\quad \mbox{for any $x\in D_n$, $n\geq0$},
	\end{align}
	where $n_0 \geq 2$ and $0<\delta <1 $ are to be chosen later.
	
Under the choice of $\mu$ in (\ref{mu3}),
it is easy to verify that (\ref{e-vol-3}) holds.
 For brevity, let us denote $\beta=\frac{1}{q-1}$, $\sigma=\frac{2-q}{q-1}$.

 Let us deal with cases of  $n = 0$ and $n \geq1$.

	\textbf{Case of $n = 0$}. Substituting (\ref{u3}) and (\ref{mu3}) into (\ref{en0}), we have
	\begin{align*}
		&\frac{\delta}{(n_0+1)^{\sigma}(\ln{(n_0+1)})^{\beta}}-\frac{\delta}{n_0^{\sigma}(\ln{n_0})^{\beta}}+\left( \frac{\delta}{n_0^{\sigma}(\ln{n_0})^{\beta}}+1\right) ^p \nonumber \\&
		\times \left( \frac{1}{2}\left( \frac{\delta}{n_0^{\sigma}(\ln{n_0})^{\beta}}-\frac{\delta}{(n_0+1)^{\sigma}(\ln{(n_0+1)})^{\beta}}\right) ^2\right) ^{\frac{q}{2}} \leq 0,
	\end{align*}
	such $\delta$ is accessible by letting $\delta \leq \delta_0$, where
	\begin{align*}
		\delta_0&= 2^{\frac{q}{2(q-1)}} \left(\frac{1}{n_0^{\sigma}(\ln{n_0})^{\beta}}-\frac{1}{(n_0+1)^{\sigma}(\ln{(n_0+1)})^{\beta}}\right)^{\frac{1}{q-1}} \\ &
		\times \left( \frac{1}{n_0^{\sigma}(\ln{n_0})^{\beta}}+1\right) ^{\frac{-p}{q-1}}\left( \frac{1}{n_0^{\sigma}(\ln{n_0})^{\beta}}-\frac{1}{(n_0+1)^{\sigma}(\ln{(n_0+1)})^{\beta}}\right)^{-\frac{q}{q-1}},
	\end{align*}
	where we used $\delta<1$, $p<0$.
	
	\textbf{Case of $n\geq1$}. By substituting (\ref{u3}) and (\ref{mu3}), (\ref{en}) is equivalent to
	\begin{align*} &\frac{\delta\frac{(n+n_0)^{\beta}(\ln{(n+n_0)})^{\beta+\epsilon}}{(n+n_0+1)^{\sigma}(\ln{(n+n_0+1)})^{\beta}}+\delta\frac{(n+n_0-1)^{\beta}(\ln{(n+n_0-1)})^{\beta+\epsilon}}{(n+n_0-1)^{\sigma}(\ln{(n+n_0-1)})^{\beta}}}{(n+n_0)^{\beta}(\ln{(n+n_0)})^{\beta+\epsilon}+(n+n_0-1)^{\beta}(\ln{(n+n_0-1)})^{\beta+\epsilon}}	\nonumber \\ &
		-\frac{\delta}{(n+n_0)^{\sigma}(\ln{(n+n_0)})^{\beta}}+\left(\frac{\delta}{(n+n_0)^{\sigma}(\ln{(n+n_0)})^{\beta}} +1\right)^p \nonumber \\ &
		\times\delta^q\left( \frac{(n+n_0)^{\beta}(\ln{(n+n_0)})^{\beta+\epsilon}\left(\frac{1}{(n+n_0)^{\sigma}(\ln{(n+n_0)})^{\beta}}-\frac{1}{(n+n_0+1)^{\sigma}(\ln{(n+n_0+1)})^{\beta}} \right)^2 }{2(n+n_0)^{\beta}(\ln{(n+n_0)})^{\beta+\epsilon}+2(n+n_0-1)^{\beta}(\ln{(n+n_0-1)})^{\beta+\epsilon}}\right.\nonumber \\& \left.		+\frac{(n+n_0-1)^{\beta}(\ln{(n+n_0-1)})^{\beta+\epsilon}\left(\frac{1}{(n+n_0-1)^{\sigma}(\ln{(n+n_0-1)})^{\beta}}-\frac{1}{(n+n_0)^{\sigma}(\ln{(n+n_0)})^{\beta}} \right)^2 }{2(n+n_0)^{\beta}(\ln{(n+n_0)})^{\beta+\epsilon}+2(n+n_0-1)^{\beta}(\ln{(n+n_0-1)})^{\beta+\epsilon}}
		\right)^\frac{q}{2}\leq0,
	\end{align*}
	which is  equivalent to
	\begin{equation*}\label{3-3-1}
		\delta^{q-1}\leq \Lambda_3(n+n_0),
	\end{equation*}
where
\begin{eqnarray*}\label{3-3-2}
	\Lambda_3(n):=&\left( 1-\frac{n^{\sigma}(\ln{n})^{\beta}\left( \frac{n^{\beta}(\ln{n})^{\beta+\epsilon}}{(n+1)^{\sigma}(\ln{(n+1)})^{\beta}}+\frac{(n-1)^{\beta}(\ln{(n-1)})^{\beta+\epsilon}}{(n-1)^{\sigma}(\ln{(n-1)})^{\beta}}\right) }{n^{\beta}(\ln{n})^{\beta+\epsilon}+(n-1)^{\beta}(\ln{(n-1)})^{\beta+\epsilon}}\right) \\ &
		\times \left(\frac{1}{(n)^{\sigma}(\ln{(n)})^{\beta}} +1\right)^{-p}((n)^{\sigma}(\ln{(n)})^{\beta})^{-1} \\ &
		\times\left( \frac{n^{\beta}(\ln{n})^{\beta+\epsilon}\left(\frac{1}{n^{\sigma}(\ln{n})^{\beta}}-\frac{1}{(n+1)^{\sigma}(\ln{(n+1)})^{\beta}} \right)^2 }{2n^{\beta}(\ln{n})^{\beta+\epsilon}+2(n-1)^{\beta}(\ln{(n-1)})^{\beta+\epsilon}}\right.\\& \left.
		+\frac{(n-1)^{\beta}(\ln{(n-1)})^{\beta+\epsilon}\left(\frac{1}{(n-1)^{\sigma}(\ln{(n-1)})^{\beta}}-\frac{1}{n^{\sigma}(\ln{n})^{\beta}} \right)^2 }{2n^{\beta}(\ln{n})^{\beta+\epsilon}+2(n-1)^{\beta}(\ln{(n-1)})^{\beta+\epsilon}}
		\right)^{-\frac{q}{2}}.
	\end{eqnarray*}
	By the similiar computation of the proof of Theorem \ref{thm2} (I), we have that
	\begin{align*}
		\lim_{n\to \infty} \Lambda_3(n)=(\frac{2}{\sigma})^{\frac{q}{2}-1} \epsilon.
	\end{align*}
This implies that there exists some large $n_0$ such that for all $n\geq 0$, the RHS of (\ref{3-3-1}) is bounded from above by $\delta_1^p$, where $\delta_1=\left[\frac{(\frac{2}{\sigma})^{\frac{q}{2}-1} \epsilon}{2}\right]^{\frac{1}{q-1}}$.
	
	Finally, choosing $n_0$ as above and $\delta=\min\{\delta_0, \delta_1\}$, we obtain that when $(p, q)\in G_3$, $u$ is a solution to (\ref{ieq}). Hence we complete the proof for Theorem \ref{thm2} (III).
\end{proof}
\vskip1ex

\begin{proof}[\rm{Proof of Theorem \ref{thm2} (IV)}]
	When $(p, q)\in G_4$, we take $\mu$ and $u$ as follows
	\begin{equation}\label{mu4}
		\mu_{xy}=\mu_n=\frac{\lambda e^{\lambda(n+n_0)}}{(N-1)^n},\quad \mbox{for any $(x,y)\in E_n$, $n\geq0$},
	\end{equation}
	\begin{equation}\label{u4}
		u(x)=u_n=\frac{\delta}{(n+n_0)}+\delta,
		\quad \mbox{for any $x\in D_n$, $n\geq0$},
	\end{equation}
	where $n_0 \geq 2$  and $\delta>0 $ are to be determined later.
	
	Under the above choices of $\mu$, it follows that (\ref{e-vol-4}) holds.
	
%
	\textbf{Case of $n = 0$}. By substituting (\ref{u4}) and (\ref{mu4}), (\ref{en0}) is equivalent to
	\begin{align*}
		\frac{\delta}{n_0+1}-\frac{\delta}{n_0}+\left( \frac{\delta}{n_0}+\delta\right) ^p
		\times \left( \frac{1}{2}\left( \frac{\delta}{n_0}-\frac{\delta}{n_0+1}\right) ^2\right) ^{\frac{1}{2}} \leq 0,
	\end{align*}
	which is satisfied by choosing $\delta \geq\delta_0$ with
	\begin{align*}
		\delta_0&= 2^{\frac{1}{2p}} \left(\frac{1}{n_0^{\sigma}(\ln{n_0})^{\beta}}-\frac{1}{(n_0+1)^{\sigma}(\ln{(n_0+1)})^{\beta}} \right)^{\frac{1}{p}} \\ &
		\times \left( \frac{1}{n_0^{\sigma}(\ln{n_0})^{\beta}}+1\right) ^{-1}\left( \frac{1}{n_0^{\sigma}(\ln{n_0})^{\beta}}-\frac{1}{(n_0+1)^{\sigma}(\ln{(n_0+1)})^{\beta}}\right)^{-\frac{1}{p}}.
	\end{align*}
	
	\textbf{Case of $n\geq1$}. By substituting (\ref{u4}) and (\ref{mu4}), (\ref{en}) is equivalent to
	\begin{align*}
		&\frac{e^{\lambda(n+n_0)}  \frac{\delta}{n+n_0+1}+e^{\lambda(n+n_0-1)} \frac{\delta}{n+n_0-1}}{e^{\lambda(n+n_0)} +e^{\lambda(n+n_0-1)} }-\frac{\delta}{n+n_0}\nonumber\\+(\frac{\delta}{n+n_0}+\delta)^p(&\frac{e^{\lambda(n+n_0)} (\frac{\delta}{n+n_0+1}-\frac{\delta}{n+n_0})^2+e^{\lambda(n+n_0-1)} (\frac{\delta}{n+n_0-1}-\frac{\delta}{n+n_0})^2}{2e^{\lambda(n+n_0)} +2e^{\lambda(n+n_0-1)} })^{\frac{1}{2}} \leq 0,
	\end{align*}
	which is  equivalent to that $\delta$ satisfies
\begin{equation}\label{3-4-1}
	\delta^{p}\leq \Lambda_4(n+n_0),
\end{equation}
	where
	\begin{align*}
		\Lambda_4(n):=&\left( \frac{1}{n}-\frac{e^{\lambda n}  \frac{1}{n+1}+e^{\lambda(n-1)} \frac{1}{n-1}}{e^{\lambda n} +e^{\lambda(n-1)} }\right)(\frac{1}{n}+1)^{-p} \nonumber\\&
		\times\left(\frac{e^{\lambda n} (\frac{1}{n+1}-\frac{1}{n})^2+e^{\lambda(n-1)} (\frac{1}{n-1}-\frac{1}{n})^2}{2e^{\lambda n} +2e^{\lambda(n-1)} }\right) ^{-\frac{1}{2}}.
	\end{align*}
	Since
	\begin{equation*}
		\lim_{n\to\infty}\Lambda_4(n)= \sqrt{2}\frac{1-e^{-\lambda}}{1+e^{-\lambda}}.
	\end{equation*}
	we obtain that there exists some large $n_0$ such that for all $n\geq 0$, the RHS of (\ref{3-4-1}) is bounded.
Letting $\delta=\max\{\delta_0, \delta_1\}$,
we derive that when $(p, q)\in G_4$, $u$ is a solution to (\ref{ieq}), where $\delta_1:=\left[\frac{\sqrt{2}\frac{1-e^{-\lambda}}{1+e^{-\lambda}}}{2}\right]^{\frac{1}{p}}$. Hence we complete the proof for Theorem \ref{thm2} (IV).
\end{proof}
\vskip1ex
\begin{proof}[\rm{Proof of Theorem \ref{thm2} (V)}]	We divide the proof into two cases:
	\begin{enumerate}
		\item[(V-1).]{$(p,q)\in \{p+q=1,p\geq 0,q>0\};$}
		\item[(V-2).]{$(p,q)\in \{p+q=1,p>1, q<0\};$}
	\end{enumerate}
	In case (V-1), we take $\mu$ and $u$ as follows
	\begin{align}\label{mu5-1}
		&\mu_{xy}=\mu_n=\frac{\lambda e^{\lambda n}}{(N-1)^n},\quad \mbox{for any $(x,y)\in E_n$, $n\geq0$}
	\end{align}
	\begin{align}\label{u5-1}
		&u(x)=u_n=e^{-\frac{\lambda}{4} n},
		\quad \mbox{for any $x\in D_n$, $n\geq0$}.
	\end{align}
	Under these choices of $\mu$, we obtain that (\ref{e-vol-5}) holds. We choose $\lambda$ later.
	
	
	\textbf{Case of $n = 0$}. Combining (\ref{u5-1}) and (\ref{mu5-1}), (\ref{en0}) is equivalent to
	\begin{align}\label{3-5-1}
		1\leq 2^{\frac{q}{2}}(1-e^{-\frac{\lambda}{4}})^{-q}(1-e^{-\frac{\lambda}{4}})= 2^{\frac{q}{2}}(1-e^{-\frac{\lambda}{4}})^{1-q},
	\end{align}
	
	\textbf{Case of $n\geq1$}. Combining (\ref{u5-1}) and (\ref{mu5-1}), (\ref{en}) is equivalent to
	\begin{align}\label{3-5-2}
		1&\leq
2^{\frac{q}{2}}(1-e^{-\frac{\lambda}{4}})^{1-q}\left(\frac{1+e^{-\lambda}}{1+e^{-\frac{\lambda}{2}}} \right) ^{\frac{q}{2}}\left(\frac{1-e^{\frac{3\lambda}{4}}}{1+e^{-\lambda}} \right).
	\end{align}
	It is easy to see that both of (\ref{3-5-1}) and (\ref{3-5-2}) hold by choosing some large $\lambda$. Hence $u$ is a solution to (\ref{ieq}).
	
	In case (V-2), Fix an arbitrary vertex $o\in T_N$ as the root, we choose a special vertex $p$, $p\sim o$. Let us define
	\begin{align*}
		&P:=\{x\in T_N|\mbox{o is not in the path between x and p}\},\\
		&D'_{-n}:=\{x\in  P|d(o, x)=n\} \qquad\mbox{for $n\geq 0$},\\
		&D'_{n}:=\{x\in T_N\setminus P|d(o, x)=n\}\qquad\mbox{for $n\geq 0$}.
	\end{align*}
	At last, we denote by $E'_n$ the collection of all the edges from vertices in $D'_n$ to
	vertices in $D'_{n+1}$ for $n\in Z$.
	
	Take $\mu$ and $u$ as follows
	\begin{align}\label{mu5-2-1}
		&\mu_{xy}=\mu_n=\frac{\lambda e^{\lambda n}}{(N-1)^n},\quad \mbox{for any $(x,y)\in E'_n$, $n\geq0$},
	\end{align}
	\begin{align}\label{mu5-2-2}
		&\mu_{xy}=\mu_n=\frac{\lambda e^{\lambda n}}{(N-1)^{-n-1}},\quad \mbox{for any $(x,y)\in E'_n$, $n\leq 1$},
	\end{align}
	\begin{align}\label{u5-2}
		&u(x)=u_n=e^{-(\lambda-1) n},
		\quad \mbox{for any $x\in D'_n$, $n\in Z$}.
	\end{align}
	Noticing $D_n=D'_{-n}\cup D'_{n}$, for $n \geq 2$, we have
	$$\mu(B(o,n))=\sum\limits^n_{k=0}\mu(D_k)\asymp \sum\limits^n_{k=0}(N-1)^k (\mu_k+\mu_{-k})\asymp  e^{\lambda n},$$
	where $\lambda>0$ is to be chosen later.
	
	Then we check that (\ref{ieq}) holds for the $\mu$ and $u$ given as above: for any $n\geq 0$,
	\begin{align}\label{3-5-2-1}
		&\frac{(N-1)\mu_n u_{n+1}+\mu_{n-1} u_{n-1}}{(N-1)\mu_n+\mu_{n-1}}-u_n\nonumber\\&+u_n^p(\frac{(N-1)\mu_n(u_{n+1}-u_n)^2+\mu_{n-1}(u_{n-1}-u_n)^2}{2(N-1)\mu_n+2\mu_{n-1}})^{\frac{q}{2}} \leq 0,
	\end{align}
	and for any $n\leq -1$,
	\begin{align}\label{3-5-2-2}
		&\frac{\mu_n u_{n+1}+(N-1)\mu_{n-1} u_{n-1}}{\mu_n+(N-1)\mu_{n-1}}-u_n\nonumber\\&+u_n^p(\frac{\mu_n(u_{n+1}-u_n)^2+(N-1)\mu_{n-1}(u_{n-1}-u_n)^2}{2\mu_n+2(N-1)\mu_{n-1}})^{\frac{q}{2}} \leq 0,
	\end{align}
hold.

Combining  (\ref{mu5-2-1})-(\ref{u5-2}), we know (\ref{3-5-2-1}) and (\ref{3-5-2-2}) are equivalent to
	\begin{align*}
		&\frac{e^{\lambda n} e^{-(\lambda-1) (n+1)}+e^{\lambda (n-1)} e^{-(\lambda-1) (n-1)}}{e^{\lambda n} + e^{\lambda (n-1)}}- e^{-(\lambda-1) (n)} \nonumber\\&+ e^{-(\lambda-1) p n}\left(\frac{e^{\lambda n} (e^{-(\lambda-1) (n+1)}- e^{-(\lambda-1)n})^2+ e^{\lambda (n-1)} (e^{-(\lambda-1) (n-1)}- e^{-(\lambda-1) n})^2}{2e^{\lambda n} +2 e^{\lambda (n-1)} }\right)^{\frac{q}{2}}\\&
		\leq 0,\quad\mbox{ for any $n\in \mathbb{Z}$},
	\end{align*}
	namely
	\begin{align}\label{3-5-2-3}
		1&\leq
		(1-e^{-(\lambda-1)})^{1-q}\left(\frac{2+e^{-\lambda}}{1+e^{\lambda-2}} \right) ^{\frac{q}{2}}\left(\frac{1-e^{-1}}{1+e^{-\lambda}} \right).
	\end{align}
	Noticing $q<0$, we obtain RHS of (\ref{3-5-2-3}) tends to infinity as $\lambda\to\infty$. This implies that there exists some large $\lambda$, such that $ u $ is a solution to (\ref{ieq}).
	Hence we complete the proof for Theorem \ref{thm2} (V).
\end{proof}


\begin{thebibliography}{99}
	%
	%
	%
	
	%
	%
	\bibitem{BHLLMY}F. Bauer, P. Horn, Y. Lin, G. Lippner,
		D. Mangoubi, and S.-T. Yau, Li-Yau inequality on graphs, Journal of Differential Geometry.  \textbf{99} (2015) 359-405.
	
	
	\bibitem{CM}F. Camilli, C. Marchi, A note on Kazdan-Warner equation on networks, Advances in Calculus of Variations (2020),  DOI: 10.1515/ACV-2020-0046.
%
	
	\bibitem{CY} S.Y. Cheng, S.-T. Yau,
		Differential equations on Riemannian manifolds and their geometric applications, Comm.
	Pure Appl. Math. \textbf{28} (1975) 333-354.
	
	%
%
%
	\bibitem{GHJ}H. Ge, B. Hua, W. Jiang, A note on Liouville type equations on graphs,  Proc. Amer. Math. Soc. \textbf{146} (2018), no. 11, 4837-4842.
%
\bibitem{G85} A. Grigor'yan, On the existence of positive
		fundamental solution of the Laplace equation on Riemannian manifolds,
	Matem. Sb. \textbf{128} (1985) 354-363. English transl. Math. USSR Sb.
	\textbf{56} (1987) 349-358.
	
	
	
	
	\bibitem{G2} \textsc{A. Grigor'yan}, Introduction to Analysis on Graphs, a textbook, AMS University Lecture Series \textbf{71}, 2018.
	%
	%
%
\bibitem{GLY2} A. Grigor'yan, Y. Lin, Y. Yang, Kazdan-Warner equations on graphs,
	Calc. Var. Partial Differential Equations 55 (2016), no. 4, Art. 92, 13 pp.
%
\bibitem{GLY3} A. Grigor'yan, Y. Lin, Y. Yang,  Existence of positive solutions to some nonlinear equations on locally finite graphs, Sci. China Math. \textbf{60}(2017), no. 7, 1311-1324.
	
	\bibitem{GS1} A. Grigor'yan, Y. Sun, On nonnegative
		solutions of the inequality $\Delta u+u^{\sigma} \le 0$ on Riemannian
		manifolds, Comm. Pure Appl. Math. \textbf{67} (2014), 1336-1352.
	
	
	
	\bibitem{GSH}Q. Gu, Y. Sun, and X. Huang, Superlinear elliptic inequalities on weighted graphs, arXiv:2201.06397.
	
	\bibitem{HSZ}X. Han, M. Shao, L. Zhao, Existence and convergence of solutions for nonlinear biharmonic equations on graphs,
	J. Differential Equations \textbf{268} (2020), no. 7, 3936-3961.
%
%
%
\bibitem{HWY} H.-Y. Huang, J. Wang, W. Yang, Mean field equation and relativistic Abelian Chern-Simons model on finite graphs, J. Funct. Anal. \textbf{281}(2021), no. 10, Paper No. 109218.
%
	\bibitem{K}L. Karp, Subharmonic functions, harmonic mappings and isometric immersions, in: ``Seminar
	on Differential Geometry'', ed. S.T.Yau, Ann. Math. Stud. \textbf{102}, Princeton, 1982.
	
	
	
%
	
	
%
	\bibitem{LiuY}S. Liu, Y. Yang, Multiple solutions of Kazdan-Warner equation on graphs in the negative case,
	Calc. Var. Partial Differential Equations \textbf{59}(2020), no. 5, Paper No. 164, 15 pp.
%
	
	%
	
	%
	%
	
	
	%
	%
	
	%
	
	
	
	
	
	
\bibitem{SXX}Y. Sun, J. Xiao, and F. Xu, A sharp Liouville principle for $\Delta_mu+u^p|\nabla u|^q\le 0$ on geodesically complete noncompact Riemannian manifolds, Math. Ann., Doi: 10.1007/s00208-021-02311-6.

	\bibitem{V} N. Varopoulos, Potential theory and diffusion
		on Riemannian manifolds, in Conf. on Harmonic Analysis in Honor of A.
	Zygmund, Wadsworth Math. Series, Wadsworth, Belmont, CA, pp. 821-837, 1983.
	
	\bibitem{W} W. Woess, Random walks on infinite graphs and groups, Cambridge U. Process, Cambridge, 2000.
	%
\end{thebibliography}
\end{document}